\documentclass[a4paper,reqno]{amsart}

\usepackage{amssymb}

\newtheorem{theorem}{Theorem}
\newtheorem{lemma}[theorem]{Lemma}

\newtheorem{proposition}[theorem]{Proposition}

\numberwithin{equation}{subsection}
\numberwithin{theorem}{subsection}

\theoremstyle{definition}

\newtheorem*{example*}{Example}

\newtheorem{remark}[theorem]{Remark}
\newtheorem*{remark*}{Remark}

\newcommand{\bF}{{\mathbb F}}
\newcommand{\bZ}{{\mathbb Z}}

\newcommand{\bH}{{\mathbb H}}

\newcommand{\cJ}{{\mathcal J}}
\newcommand{\cC}{{\mathcal C}}
\newcommand{\cD}{{\mathcal D}}
\newcommand{\cK}{{\mathcal K}}
\newcommand{\cR}{{\mathcal R}}
\newcommand{\cQ}{{\mathcal Q}}
\newcommand{\cA}{{\mathcal A}}

\newcommand{\frg}{{\mathfrak g}}

\newcommand{\frf}{{\mathfrak f}}
\newcommand{\fre}{{\mathfrak e}}

\newcommand{\fra}{{\mathfrak a}}
\newcommand{\frb}{{\mathfrak b}}

\newcommand{\calI}{{\mathcal I}}

\newcommand{\subo}{_{\bar 0}}
\newcommand{\subuno}{_{\bar 1}}

\providecommand{\espan}[1]{\text{span}\left\{ #1\right\}}
\providecommand{\alg}[1]{\text{alg}\left\langle #1\right\rangle}

\providecommand{\spmatrix}[1]{\left(\begin{smallmatrix} {#1}&0\\ 0&{#1}^{-1}\end{smallmatrix}\right)}

 \DeclareMathOperator{\tkk}{\mathfrak{tkk}}
 
 \DeclareMathOperator{\frsl}{{\mathfrak{sl}}}
 \DeclareMathOperator{\frsp}{{\mathfrak{sp}}}
 \DeclareMathOperator{\frso}{{\mathfrak{so}}}
 
 \DeclareMathOperator{\frgl}{{\mathfrak{gl}}}
 
 \DeclareMathOperator{\frosp}{{\mathfrak{osp}}}

 \DeclareMathOperator{\ad}{ad}

 \DeclareMathOperator{\der}{\mathfrak{der}}
 
 \DeclareMathOperator{\End}{End}
 \DeclareMathOperator{\Hom}{Hom}
 \DeclareMathOperator{\Mat}{Mat}
 \DeclareMathOperator{\Aut}{Aut}
  \DeclareMathOperator{\Int}{Inn}

 \DeclareMathOperator{\Diag}{Diag}
 
   \DeclareMathOperator{\Supp}{Supp}
 \DeclareMathOperator{\degree}{deg}

 \DeclareMathOperator{\Cl}{\mathfrak{Cl}}

\newenvironment{romanenumerate}
 {\begin{enumerate}
 
 }{\end{enumerate}}

\begin{document}

\title{Fine gradings on exceptional simple Lie superalgebras}


\author[Cristina Draper]{Cristina Draper$^{\star}$}
\thanks{$^{\star}$ Supported by the the Spanish Ministerio de
 Educaci\'on y Ciencia
 and FEDER  (MTM2007-60333 and MTM2010-15223) and by the Junta de
Andaluc\'{\i}a and Fondos FEDER through projects FQM-336,
FQM-3737 and FQM-2467 and by the Spanish Ministerio de Educaci\'on y Ciencia
under the project ``Ingenio Mathematica (i-math'' No. CSD2006-00032,
Consolider-Ingenio 2010)}

\address{Departamento de Matem\'atica Aplicada, Escuela de las Ingenier\'{\i}as, Universidad de M\'alaga, Ampliaci\'on del Campus de Teatinos,  29071 M\'alaga, Spain}
\email{cdf@uma.es}

\author[Alberto Elduque]{Alberto Elduque$^{\star\star}$}
\thanks{$^{\star\star}$ Supported by the Spanish Ministerio de
 Educaci\'on y Ciencia
 and FEDER (MTM2007-67884-C04-02 and MTM2010-18370-C04-02) and by the
 Diputaci\'on General de Arag\'on (Grupo de Investigaci\'on de
 \'Algebra)}
\address{Departamento de Matem\'aticas e
 Instituto Universitario de Matem\'aticas y Aplicaciones,
 Universidad de Zaragoza, 50009 Zaragoza, Spain}
\email{elduque@unizar.es}

\author[C\'andido Mart\'{\i}n-Gonz\'alez]{C\'andido Mart\'{\i}n-Gonz\'alez$^{\star}$}
\address{Departamento de \'Algebra, Geometr\'{\i}a y Topolog\'{\i}a, Facultad de Ciencias, Universidad de M\'alaga, Campus de Teatinos, 29080 M\'alaga, Spain}
\email{candido@apncs.cie.uma.es}





\begin{abstract}
The fine abelian group gradings on the simple exceptional classical Lie superalgebras over algebraically closed fields of characteristic $0$ are determined up to equivalence.
\end{abstract}

\maketitle

\section{Introduction}\label{se:Introduction}

The purpose of this paper is the complete determination of the fine gradings on the exceptional simple Lie superalgebras over an algebraically closed field of characteristic $0$ in Kac's classification \cite{KacLie}: $D(2,1;\alpha)$ ($\alpha\ne 0,-1$), $G(3)$ and $F(4)$.

Let $\cA$ be an algebra (not necessarily associative) over a field $\bF$. An \emph{abelian group grading} on $\cA$ is a decomposition
\begin{equation}\label{eq:grading}
\Gamma:\cA=\oplus_{g\in G}\cA_g
\end{equation}
into a direct sum of subspaces (the \emph{homogeneous components}), where $G$ is an abelian group, and such that $\cA_g\cA_h\subseteq \cA_{g+h}$ for any $g,h\in G$. Since only abelian group gradings will be considered in this paper, we will refer to them simply as \emph{gradings}.

Given two gradings $\Gamma:\cA=\oplus_{g\in G}\cA_g$ and $\Gamma':\cA=\oplus_{g'\in G'}\cA_{g'}$, $\Gamma$ is said to be a \emph{refinement} of $\Gamma'$ (or $\Gamma'$ a \emph{coarsening} of $\Gamma$) if for any $g\in G$ there is  $g'\in G'$ such that $\cA_g\subseteq \cA_{g'}$. That is, any homogeneous component of $\Gamma'$ is the direct sum of homogeneous components of $\Gamma$. The refinement is proper if there are $g\in G$ and $g'\in G'$ such that $\cA_g\subsetneq\cA_{g'}$. A grading is said to be \emph{fine} if it admits no proper refinements. Therefore, any grading is obtained as a coarsening of some fine grading.

Moreover, the gradings $\Gamma$ and $\Gamma'$ are said to be \emph{equivalent} if there is an automorphism $\varphi\in \Aut\cA$ such that for any $g\in G$ there is a $g'\in G'$ with $\varphi(\cA_g)=\cA_{g'}$.

The \emph{type} of a grading $\Gamma$ on a finite-dimensional algebra $\cA$ is the sequence of numbers $(n_1,\ldots,n_r)$ where $n_i$ is the number of homogeneous components of dimension $i$, $i=1,\ldots,r$, with $n_r\ne 0$. Thus $\dim \cA=\sum_{i=1}^r in_i$.

\smallskip

Assume from now on that the ground field $\bF$ is algebraically closed of characteristic $0$ and that $\cA$ is finite-dimensional over $\bF$. Given any grading $\Gamma$, the \emph{diagonal group} $\Diag_\Gamma(\cA)$ of $\Gamma$ is the subgroup of $\Aut\cA$ consisting of those automorphisms which act as a scalar multiple of the identity on each homogeneous component.
If $\Gamma$ is a fine grading, then $\Diag_\Gamma(\cA)$ is a maximal abelian diagonalizable subgroup of $\Aut\cA$ (see \cite{PateraZassenhaus}). And conversely, given any maximal abelian diagonalizable subgroup $M$ of $\Aut\cA$, the algebra $\cA$ decomposes as the direct sum of the common eigenspaces of the elements in $M$: $\Gamma_M:\cA=\oplus_{\chi\in\hat M}\cA_{\chi}$, where $\hat M$ is the group of characters of $M$, that is, the continuous group homomorphisms $\chi:M\rightarrow \bF^\times$ for the Zariski topology, with $\cA_\chi=\{ x\in\cA: \varphi(x)=\chi(\varphi)x\ \forall\varphi\in M\}$. Then $\Gamma_M$ is a fine grading. In this way, the classification of fine gradings on $\cA$ up to equivalence corresponds to the classification of maximal abelian diagonalizable subgroups of $\Aut\cA$ up to conjugation.


\smallskip

A very important fine grading on a simple finite-dimensional Lie algebra $\frg$ over $\bF$ is given by the root space decomposition relative to a Cartan subalgebra (the \emph{Cartan grading}). This is a grading over $\bZ^r$ with $r$ the rank of the Lie algebra. But there appear many other fine gradings, some of them related to gradings on matrix algebras. The fine gradings on the simple classical Lie algebras have been determined through the work of many authors (see \cite{Eld.Fine} and the references therein). For the exceptional simple Lie algebras, the fine gradings have been determined for $G_2$ (see \cite{CristinaCandidoG2}, \cite{BahturinTvalavadzeG2}, or \cite{EKG2F4}), for $F_4$ (\cite{CristinaCandidoF4}, \cite{CristinaF4revisited}, or \cite{EKG2F4}), and for $E_6$ (\cite{CristinaViruE6}). Many of the fine gradings on these exceptional algebras are strongly related to gradings on either the octonions or the exceptional simple Jordan algebra (or Albert algebra).

\medskip

The definitions above on gradings on algebras carry over in a straightforward way to superalgebras. Thus a grading on a superalgebra $\cA=\cA\subo\oplus\cA\subuno$ is a decomposition as in \eqref{eq:grading} with the homogeneous components $\cA_g$'s being subspaces in the super sense: $\cA_g=(\cA_g\cap \cA\subo)\oplus(\cA_g\cap\cA\subuno)$. It makes sense to talk about refinements, fine gradings, equivalent gradings, ...

\smallskip

The exceptional simple Lie superalgebras $D(2,1;\alpha)$ ($\alpha\ne 0,-1$), $G(3)$ and $F(4)$ in Kac's classification \cite{KacLie} are related to some noteworthy algebras and superalgebras. The Lie superalgebras $D(2,1;\alpha)$ ($\alpha\ne 0,-1$) are related to the algebra of quaternions, while $G(3)$ and $F(4)$ are related to the octonions. On the other hand, $F(4)$ appears related to an exceptional simple ten-dimensional Jordan superalgebra discovered by Kac \cite{KacJordan} by means of the so-called Tits-Kantor-Koecher construction. The gradings on the octonions were classified in \cite{EldGrOctonions}, and the gradings on the Kac Jordan superalgebra in \cite{AntonioCristinaCandidoK10}. These results will be used here.

\smallskip

The paper is structured as follows. The next section will be devoted to review some facts on gradings on matrix algebras, as some gradings on the even Clifford algebra of a seven-dimensional vector space endowed with a nondegenerate quadratic form will be needed in dealing with the simple Lie superalgebra $F(4)$. The relationships of $F(4)$ with the octonions and with Kac Jordan superalgebra will be reviewed in Section~\ref{se:modelsF(4)}. These results will be instrumental in the classification of the gradings on $F(4)$ in Section~\ref{se:gradingsF(4)}. Section~\ref{se:gradingsG(3)} deals with the fine gradings on $G(3)$. The situation here is simpler using the known results on gradings on the octonions and on the simple Lie algebra of type $G_2$. In Section~\ref{se:AutD21a} we will look at certain subgroups of the group of automorphisms of $D(2,1;\alpha)$ ($\alpha\ne 0,-1$) which will appear in Section~\ref{se:innergradingsD21a}, which deals with those fine gradings on these Lie superalgebras with corresponding maximal abelian diagonalizable subgroup fixing the three simple ideals of the even part. Finally Section~\ref{se:outerD21a} will deal with the remaining cases in which there are automorphisms in the maximal abelian diagonalizable subgroup which permute some of these simple ideals.

\smallskip

It turns out that, up to equivalence, there are five fine gradings on $F(4)$ (Theorem \ref{th:MainThF(4)}), only two on $G(3)$ (Theorem \ref{th:MainThG(3)}), and five on $D(2,1;\alpha)$ ($\alpha\ne 0,-1$) if $\alpha\not\in\{1,-2,-\frac12,\omega,\omega^2\}$ (where $\omega$ is a primitive cubic root of $1$), six on $D(2,1;\alpha)$ if $\alpha\in\{1,-2,-\frac12\}$, and four on $D(2,1;\alpha)$ if $\alpha\in \{\omega,\omega^2\}$ (Theorem \ref{th:MainThD21a}).

In what follows our ground field $\bF$ will be assumed to be algebraically closed and of characteristic $0$. All vector spaces and (super)algebras will be  finite-dimensional over $\bF$, and unadorned tensor products $\otimes$ will be assumed to be defined over $\bF$.

\bigskip


\section{Gradings on matrix algebras}\label{se:Matrix}

In this section we will recall some known results which can be consulted in \cite{Eld.Fine}.

\subsection{}\label{ss:R} Let ${\cR}$ be a simple central associative algebra (that is, ${\cR}$ is isomorphic to the algebra of matrices $\Mat_n(\bF)$ for some natural number $n$), and let $\Gamma: {\cR}=\oplus_{g\in G} {\cR}_g$ be a grading on ${\cR}$.  Let $I$ be a minimal left graded ideal of ${\cR}$. Then there is an idempotent $e\in {\cR}_0$ such that $I={\cR}e$, the subalgebra ${\cD}=e{\cR}e$ is a graded division algebra isomorphic (as graded algebras) to the endomorphism ring $\End_{\cR}(I)$ (action of ${\cD}$ on $I$ on the right), where $\End_{\cR}(I)$ is endowed with the grading induced by the grading on $I$.

Let $V$ be a graded irreducible module for ${\cR}$: $V=\oplus_{g\in \tilde G} V_g$, where $\tilde G$ is a group containing the group $G$ as a subgroup, and where ${\cR}_hV_g\subseteq V_{h+g}$ for any $h\in G$ and $g\in \tilde G$. By irreducibility $V={\cR}V=I{\cR}V=IV$, so there is a homogeneous element $v\in V$ such that $V=Iv$. Therefore $V$ is isomorphic to $I$ by means of a homogeneous map $\varphi:I\rightarrow V$, $x\mapsto xv$, of degree $\deg(v)\in \tilde G$ ($\varphi(I_h)\subseteq V_{h+\deg(v)}$). Thus the grading on ${\cR}$ determines uniquely the grading on the graded irreducible module $V$ up to a shift (by the element $\deg(v)$).

Then there appears the natural graded isomorphism $\Phi:{\cD}\simeq \End_{\cR}(I)\rightarrow \End_{\cR}(V)$, $f\mapsto \varphi f\varphi^{-1}$, between the corresponding graded division algebras. In this situation we will write $\calI({\cR})=[{\cD}]$, where $[{\cD}]$ denotes the isomorphism class of the graded division algebra ${\cD}$ (as a graded algebra).

Moreover, $\cR$ is then graded isomorphic to $\End_{\cD}(V)$, this latter algebra endowed with the grading induced by the one in $V$. In particular the homogeneous component of degree $0$ of $\cR\cong\End_\cD(V)\cong \Mat_m(\cD)$ ($m=\dim_\cD V$) contains the diagonal matrices with entries in $\cD_0=\bF$, so it contains $m$ orthogonal idempotents.

\smallskip

The structure of the graded division algebras is given in \cite[Proposition 2.1]{Eld.Fine}:

\begin{proposition}\label{pr:graded_division}
Let $\Gamma: {\cD}=\oplus_{g\in G}{\cD}_g$ be a central graded division algebra over the group $G$. Then $H=\Supp\Gamma$ is a subgroup of $G$ and there is a decomposition $H=H_1\times \cdots\times H_r$, where for each $i=1,\ldots,r$, $H_i$ is isomorphic to $\bZ_{n_i}\times\bZ_{n_i}$, $n_i$ a power of a prime, such that ${\cD}={\cD}_1\cdot\cdots\cdot {\cD}_r\simeq {\cD}_1\otimes \cdots\otimes {\cD}_r$, where each ${\cD}_i=\oplus_{h\in H_i}{\cD}_h$ is graded isomorphic to the algebra $A_{n_i}$, where $A_n$ is the algebra of matrices $\Mat_n(\bF)$, but considered as the algebra generated by two elements $x$ and $y$ satisfying $x^n=1=y^n$, $xy=\epsilon_n yx$,
where $\epsilon_n$ is a primitive $n$th root of unity, and graded over $\bZ_n\times\bZ_n$ with $\degree(x)=(\bar 1,\bar 0)$ and $\degree(y)=(\bar 0,\bar 1)$.
\end{proposition}

The most important example for us of a graded division algebra is the graded quaternion algebra ${\cQ}=\Mat_2(\bF)$, which is a $\bZ_2\times\bZ_2$-graded division algebra with:
\begin{equation}\label{eq:quaternion_units}
{\cQ}_{(\bar 0,\bar 0)}=\bF 1,\quad
 {\cQ}_{(\bar 1,\bar 0)}=\bF q_1,\quad
 {\cQ}_{(\bar 0,\bar 1)}=\bF q_2,\quad
 {\cQ}_{(\bar 1,\bar 1)}=\bF q_3,
\end{equation}
where
\[
q_1=\begin{pmatrix} 1&0\\ 0&-1\end{pmatrix},\quad
q_2=\begin{pmatrix} 0&1\\ 1&0\end{pmatrix},\quad
q_3=\begin{pmatrix} 0&1\\ -1&0\end{pmatrix},
\]
so that $q_1^2=1=q_2^2$ and $q_1q_2=-q_2q_1=q_3$. It is endowed with its standard involution, where $\bar q_i=-q_i$ for $i=1,2,3$, which preserves the grading. This involution is the adjoint involution relative to the norm $N:\cQ\rightarrow \bF$, where $N(x)=\det(x)$ for any $x\in \cQ$. Note that $\cQ=A_2$.

\smallskip

Given a graded involution $\tau$ of a simple graded algebra ${\cR}\simeq \End_{\cD}(V)$ as above (this means that ${\cR}_g^\tau \subseteq {\cR}_g$ for any $g\in G$), and a fixed graded involution $\tau_{\cD}$ of the graded division algebra ${\cD}$, then there is $\epsilon=\pm 1$ and an $\epsilon$-hermitian form $h:V\times V\rightarrow {\cD}$ relative to $\tau_{\cD}$:
\[
h(v,wd)=h(v,w)d,\quad h(w,v)=\epsilon h(v,w)^{\tau_{\cD}},
\]
for any $d\in D$ and $v,w\in V$, such that $h(rv,w)=h(v,r^\tau w)$ for any $r\in {\cR}$ and $v,w\in V$. (See \cite[Section 3]{Eld.Fine}.)

\bigskip

\subsection{}\label{ss:U} Let $(U,q)$ be a vector space graded over an abelian group $G$: $U=\oplus_{g\in G} U_g$, and endowed with a nondegenerate quadratic from such that $q(U_g,U_h)=0$ unless $g+h=0$, where $q(u,v)=q(u+v)-q(u)-q(v)$ denotes the polar form of $q$. In other words, $q$ induces a degree $0$ linear form $q:U\otimes U\rightarrow \bF$, where $\bF$ is endowed with the trivial grading $\bF=\bF_0$. In this situation the grading on $U$ is said to be \emph{compatible} with the quadratic form $q$.

Assume that the dimension of $U$ is odd: $\dim U=2n+1$. Then there is a homogeneous basis $\{u_i,v_i: i=1,\ldots,m\}\cup\{w_j:j=1\ldots,2l+1\}$ with $n=m+l$,
\[
\deg(u_i)=-\deg(v_i)=g_i\in G,\quad \deg(w_j)=h_j\in G,\ 2h_j=0,
\]
for $i=1,\ldots,m$, $j=1,\ldots,2l+1$,
and the only nonzero values of $q$ for basic elements are given by
\[
q(u_i,v_i)=1=q(w_j),
\]
for any $i=1,\ldots,m$ and $j=1,\ldots,2l+1$.

This grading on $U$ induces a grading on the Clifford algebra $\Cl(U,q)$, which is the quotient of the tensor algebra on $U$ by the (graded) ideal generated by $\{u^2-q(u)1: u\in U\}$. The Clifford algebra $\Cl(U,q)$ is also naturally $\bZ_2$-graded (the elements of $U$ being odd), so $\Cl(U,q)=\Cl\subo(U,q)\oplus\Cl\subuno(U,q)$, and the even Clifford algebra $\Cl\subo(U,q)$ is central simple and inherits the grading over $G$ induced by the grading on $U$.

The element
\[
z=[u_1,v_1]\dotsm[u_m,v_m]w_1\dotsm w_{2l+1}
\]
is central with $z^2=(-1)^l$ (note that $[u_1,v_1]^2=(u_1v_1-v_1u_1)^2=u_1v_1u_1v_1+v_1u_1v_1u_1=u_1v_1+v_1u_1=q(u_1,v_1)1=1$ since $u_1^2=v_1^2=0$), and the center of $\Cl(U,q)$ is $\bF 1+\bF z$.

The next lemma will be quite useful later on.

\begin{lemma}\label{le:Uuv} Let $(U,q)$ be a $G$-graded nondegenerate odd dimensional quadratic space as before, let $u,v$ be homogeneous elements of $U$ with $q(u)=q(v)=0$ and $q(u,v)=1$, and let $U'$ be the orthogonal subspace to $u$ and $v$ (which inherits a $G$-grading too). Then there is a two-dimensional $G$-graded vector space $W$ such that $\Cl\subo(U,q)$ is isomorphic, as a graded algebra over $G$, to the tensor product $\End_{\bF}(W)\otimes\Cl\subo(U',q)$, where the grading on $\End_\bF(W)$ is the one induced by the grading on $W$.
\end{lemma}
\begin{proof}
Complete $u_1=u$ and $v_1=v$ to a homogeneous basis as above and let $z=[u_1,v_1]\dotsm [u_m,v_m]w_1\dotsm w_{2l+1}$. Then $\Cl\subo(U,q)$ is generated as an algebra by the elements $zu_1,zv_1,\ldots,zu_m,zv_m,zw_1,\ldots,zw_{2l}$. Let $\bar z=[u_2,v_2]\dotsm [u_m,v_m]w_1\dotsm w_{2l+1}$. Then $\bar z^2=(-1)^l$ too, and $z=[u_1,v_1]\bar z=\bar z[u_1,v_1]$. Besides, $\bar z$ commutes with $u_2,v_2,\ldots,u_m,v_m,w_1,\ldots,w_{2l+1}$ but $\bar z u_1=-u_1\bar z$, $\bar z v_1=-v_1\bar z$. As $[u_1,v_1]=(-1)^l[zu_1,zv_1]$, $\Cl\subo(U,q)$ is generated by the elements $zu_1$ and $zv_1$ and by the elements $\bar zu_2,\bar zv_2,\ldots,\bar zu_m,\bar zv_m,\bar zw_1,\ldots,\bar zw_{2l}$. The first two elements commute with the remaining ones, so we get:
\[
\begin{split}
\Cl\subo(U,q)&=\alg{zu_1,zv_1}\alg{\bar zu_2,\bar zv_2,\ldots,\bar zu_m,\bar zv_m,\bar zw_1,\ldots,\bar zw_{2l}}\\
&\simeq \alg{zu_1,zv_1}\otimes\alg{\bar zu_2,\bar zv_2,\ldots,\bar zu_m,\bar zv_m,\bar zw_1,\ldots,\bar zw_{2l}}\\
&=\alg{zu_1,zv_1}\otimes\Cl\subo(U',q).
\end{split}
\]
But the map $zu_1\mapsto \left(\begin{smallmatrix} 0&1\\ 0&0\end{smallmatrix}\right)$, $zv_1\mapsto\left(\begin{smallmatrix} 0&0\\ (-1)^l&0\end{smallmatrix}\right)$, extends to an isomorphism $S=\alg{zu_1,zv_1}\rightarrow \Mat_2(\bF)$. The element $e=(-1)^l(zu_1)(zv_1)=u_1v_1$ is a degree $0$ idempotent with $eSe=\bF e$ and hence $I=Se=\espan{u_1v_1,zv_1}$ is an irreducible graded left ideal. With $W=I$ we get $S\simeq \End_{\bF}(W)$ as required.
\end{proof}

\smallskip

Note that $\{u_1v_1,zv_1\}$ is a homogeneous basis of $W$ with $\deg(u_1v_1)=0$ and $\deg(zv_1)=\deg(z)+\deg(v_1)=\deg(z)-\deg(u_1)$, with $\deg(z)$ being the sum of the degrees of the elements in the homogeneous basis considered (although it is independent of the basis). The order of $\deg(z)$ is $2$.

\medskip

Take a homogeneous basis as before. Then the degree of the odd central element $z$ is $h=h_1+\cdots+h_{2l+1}$ with $2h=0$. The even Clifford algebra $\Cl\subo(U,q)$ is generated by the homogeneous elements $zu_1,zv_1,\ldots,zu_m,zv_m,zw_1,\ldots,zw_{2l}$, and hence we may shift the degrees in $U$ by assigning a new degree $\widetilde{\deg}(u)=\deg(u)+h$ to any homogeneous $u\in U$. With this new grading $\widetilde{\deg}(z)=(2l+1)h+h=0$. Since $2h=0$ this does not change the degree of the homogeneous elements $zu$, as $\widetilde{\deg}(z)+\widetilde{\deg}(u)=\deg(z)+\deg(u)$, and hence it does not change the grading on the even Clifford algebra $\Cl\subo(U,q)$ (it only adds $h$ to the degree of the homogeneous odd elements!). In the new grading, the sum of the degrees of any homogeneous basis of $U$ is $0$.

In what follows we will consider gradings on $U$ with this property.

It must be remarked here that this grading cannot be shifted while preserving its properties: if we shift the grading, that is, if we add a fixed $g$ in some abelian group containing $g_1,\ldots,g_m,h_1,\ldots,h_{2l+1}$ to the degree of the elements $u_1,v_1,\ldots,u_m,v_m,w_1,\ldots,w_{2l+1}$, then we must have $2g=0$ (as $q$ has degree $0$) and $(2l+1)g=0$ (as the sum of the degrees of the elements of any homogeneous basis is $0$). We conclude that $g=0$, so the shift is trivial.

\smallskip

The Clifford algebra $\Cl(U,q)$ is endowed with the so called \emph{bar involution} determined by the condition $\bar u=-u$ for any $u\in U$. Its restriction to $\Cl\subo(U,q)$ will play a key role here.

\smallskip

To finish this subsection, note that for $u,v,w\in U$, we have
\[
\begin{split}
[[u,v],w]&=uvw-vuw-wuv+wvu\\
&=(q(v,w)u-uwv)-(q(u,w)v-vwu)\\
&\qquad -(q(w,u)v-uwv)+(q(w,v)u-vwu)\\
&=2\bigl(q(v,w)u-q(u,w)v\bigr).
\end{split}
\]
Since the orthogonal Lie algebra $\frso(U,q)$ is spanned by the maps $w\mapsto q(u,w)v-q(v,w)u$ for $u,v\in U$, it follows that the subspace of $\Cl\subo(U,q)$ spanned by the elements $[u,v]$, $u,v\in U$ is isomorphic to $\frso(U,q)$. Hence $\frso(U,q)$ embeds, as a graded subalgebra, in the Lie algebra $\Cl\subo(U,q)^-$ defined on the associative algebra $\Cl\subo(U,q)$ by means of the Lie bracket $[x,y]=xy-yx$. Denote by $\iota:\frso(U,q)\rightarrow \Cl\subo(U,q)$ this embedding.

\bigskip

\subsection{}\label{ss:Cayleyexample} Let us consider in this and the following subsection a couple of examples in dimension $7$ which will be instrumental in our work on the fine gradings on the exceptional simple classical Lie superalgebra $F(4)$.

Let $\cC$ be the Cayley algebra over $\bF$. This algebra $\cC$ has a basis $\{1,e_i: i=1,\ldots,7\}$ with multiplication given by $1x=x1=x$ for any $x$, and
\[
e_i^2=-1\ \forall i,\quad e_ie_{i+1}=-e_{i+1}e_i=e_{i+3}
\]
and cyclically on $e_i$, $e_{i+1}$ and $e_{i+3}$ (indices modulo $7$).
Besides, $\cC$ is endowed with a nondegenerate quadratic form $N$ (the \emph{norm}) such that the above basis is orthogonal and $N(1)=N(e_i)=1$ for any $i$. This quadratic form permits composition: $N(xy)=N(x)N(y)$ for any $x,y\in\cC$, and any $x\in\cC$ satisfies the degree $2$ equation $x^2-N(x,1)x+N(x)1=0$.

The Cayley algebra is $\bZ_2^3$-graded with $\deg(e_1)=(\bar 1,\bar 0,\bar 0)$, $\deg(e_2)=(\bar 0,\bar 1,\bar 0)$, and $\deg(e_3)=(\bar 0,\bar 0,\bar 1)$.

Let $\cC^0$ be the subspace of $\cC$ orthogonal to $1$, that is, the subspace spanned by the $e_i$'s, and consider the linear map given by left multiplication in $\cC$:
\[
\begin{split}
l:\cC^0&\rightarrow \End_{\bF}(\cC)\\
x&\mapsto l_x:\cC\rightarrow \cC,\ y\mapsto xy.
\end{split}
\]
Since $l_x^2=l_{x^2}=-N(x)1$ for any $x\in \cC_0$, $l$ induces a graded homomorphism $\Cl(\cC^0,-N)\rightarrow \End_\bF(\cC)$ which restricts, by dimension count, to a graded isomorphism $\Cl\subo(\cC^0,-N)\rightarrow \End_\bF(\cC)$.

This shows $\calI\bigl(\Cl\subo(\cC_0,-N)\bigr)=[\bF]$.

Moreover, the bar involution in $\Cl(\cC^0,-N)$ is determined by $\bar x=-x$ for any $x\in \cC^0$. Under the isomorphism above, it corresponds to the involution on $\End_\bF(\cC)$ such that $\overline{l_x}=-l_x$ for any $x\in\cC^0$. But $N(xy,z)=-N(y,xz)$ for any $x\in \cC^0$ and $y,z\in\cC$. Hence the bar involution corresponds to the orthogonal involution of $\End_\bF(\cC)$ induced by the norm $N$.

\bigskip

\subsection{}\label{ss:Quaternionexample} Let us consider now the $\bZ_2^2$-graded division algebra of split quaternions ${\cQ}$ considered in \textbf{\ref{ss:R}}.

Then ${\cQ}\otimes {\cQ}$ is a natural $\bZ_2^4$-graded right module for ${\cQ}$ by means of
\[
(x\otimes y)q=x\otimes (yq),\quad \deg(x\otimes y)=\bigl(\deg(x),\deg(y)\bigr)\in\bZ_2^2\times\bZ_2^2=\bZ_2^4.
\]
Thus ${\cQ}\otimes {\cQ}$ is a free right ${\cQ}$-module of rank $4$.

Consider the standard involution on ${\cQ}$ and the linear map given by
\[
\begin{split}
\Phi: {\cQ}\otimes {\cQ}\otimes {\cQ}&\longrightarrow \End_{\cQ}({\cQ}\otimes {\cQ})\\
a\otimes b\otimes c\ &\mapsto \Phi_{a\otimes b\otimes c}: x\otimes y\mapsto (ax\bar b)\otimes (cy).
\end{split}
\]
Then $\Phi$ is a homomorphism of algebras. But ${\cQ}\otimes {\cQ}\otimes {\cQ}\simeq \Mat_8(\bF)$ is simple, so by dimension count we conclude that $\Phi$ is an algebra isomorphism, and hence $\Phi$ becomes an isomorphism of $\bZ_2^4$-graded algebras, where the grading on $\End_{\cQ}({\cQ}\otimes {\cQ})$ is the one induced by the $\bZ_2^4$-grading on ${\cQ}\otimes {\cQ}$, and the grading on ${\cQ}\otimes {\cQ}\otimes {\cQ}$ is then given by
\[
\deg(a\otimes b\otimes c)=\bigl(\deg(a)+\deg(b),\deg(c)\bigr)\in\bZ_2^2\times\bZ_2^2=\bZ_2^4,
\]
for any homogeneous $a,b,c\in {\cQ}$.

\smallskip

Consider now the subspace $U$ of ${\cQ}\otimes {\cQ}\otimes {\cQ}$ spanned by the following elements:
\begin{equation}\label{eq:wis}
\begin{aligned}
w_1&=q_1\otimes 1\otimes 1,&\quad w_2&=q_3\otimes 1\otimes 1,\\
w_3&=q_2\otimes 1\otimes q_1,&w_4&=q_2\otimes 1\otimes q_3,\\
w_5&=q_2\otimes q_1\otimes q_2,&w_6&=q_2\otimes q_3\otimes q_2,\\
w_7&=q_2\otimes q_2\otimes q_2.&&
\end{aligned}
\end{equation}
Note that $w_iw_j=-w_jw_i$  for any $i\ne j$, and $w_i^2$ is either $1$ or $-1$ for any $i=1,\ldots,7$ in the algebra ${\cQ}\otimes {\cQ}\otimes {\cQ}$.

Endow $U$ with the nondegenerate quadratic form $q$ where $\{w_i: i=1,\ldots 7\}$ becomes an orthogonal basis and $w_i^2=q(w_i)1$ for any $i$. The subspace $U$ inherits the $\bZ_2^4$-grading on ${\cQ}\otimes {\cQ}\otimes {\cQ}$. Then the inclusion $U\hookrightarrow {\cQ}\otimes {\cQ}\otimes {\cQ}$ satisfies $u^2=q(u)1$ for any $u\in U$, so it induces a homomorphism of graded algebras $\Cl(U,q)\rightarrow {\cQ}\otimes {\cQ}\otimes {\cQ}$ which is the identity on $U$, and which restricts, by dimension count, to a graded isomorphism $\Cl\subo(U,q)\rightarrow {\cQ}\otimes {\cQ}\otimes {\cQ}$.

Note that the degrees of the basic elements are:
\[
\begin{aligned}
\deg(w_1)&=(\bar 1,\bar 0,\bar 0,\bar 0),&\quad
\deg(w_2)&=(\bar 1,\bar 1,\bar 0,\bar 0),\\
\deg(w_3)&=(\bar 0,\bar 1,\bar 1,\bar 0),&
\deg(w_4)&=(\bar 0,\bar 1,\bar 1,\bar 1),\\
\deg(w_5)&=(\bar 1,\bar 1,\bar 0,\bar 1),&
\deg(w_6)&=(\bar 1,\bar 0,\bar 0,\bar 1),\\
\deg(w_7)&=(\bar 0,\bar 0,\bar 0,\bar 1).&&
\end{aligned}
\]
Hence all these degrees are different, and they sum up to $0$.

The graded isomorphism $\Phi$ shows that $\calI\bigl(\Cl\subo(U,q)\bigr)=[{\cQ}]$ in this case.

Moreover, the bar involution on $\Cl(U,q)$ is determined by the fact that $\bar w_i=-w_i$ for any $i$, and therefore it corresponds to the involution on ${\cQ}\otimes {\cQ}\otimes {\cQ}$ given by
\[
\overline{a\otimes b\otimes c}=\bar a\otimes \bar b\otimes q_2^{-1}\bar c q_2,
\]
for any $a,b,c\in {\cQ}$. (Note that $q_2^{-1}\bar q_2q_2=-q_2$, while $q_2^{-1}\bar q_iq_2=-q_2^{-1}q_iq_2=q_i$ for $i=1,3$.)

Under the isomorphism $\Phi$, this bar involution corresponds to the adjoint relative to the skew-hermitian form
\begin{equation}\label{eq:h}
\begin{split}
h:({\cQ}\otimes {\cQ})\times({\cQ}\otimes {\cQ})&\longrightarrow {\cQ}\\
\bigl( x\otimes y, u\otimes v\bigr)\,&\mapsto \bar yq_2v N(x,u).
\end{split}
\end{equation}
(Note that indeed $h\bigl(x\otimes y,(u\otimes v)q\bigr)=h\bigl(x\otimes y,u\otimes v\bigr)q$, and $h\bigl(x\otimes y,u\otimes v\bigr)=\bar yq_2v N(x,u)=-\overline{\bar v q_2 y N(u,x)}=-\overline{h\bigl(u\otimes v,x\otimes y\bigr)}$ for any $x,y,u,v,q\in {\cQ}$, since $\bar q_2=-q_2$.)

Actually, for any $a,b,c,x,y,u,v\in {\cQ}$:
\[
\begin{split}
h\Bigl(\Phi_{a\otimes b\otimes c}(x\otimes y),u\otimes v\Bigr)&=
  h\Bigl((ax\bar b)\otimes (cy), u\otimes v\Bigr)\\
  &=\bar y\bar c q_2v N(ax\bar b,u)\\
  &=\bar y q_2(q_2^{-1}\bar c q_2)v N(x,\bar a u b)\\
  &=h\Bigl(x\otimes y,(\bar a u b)\otimes (q_2^{-1}\bar cq_2)v\Bigr)\\
  &=h\Bigl(x\otimes y,\Phi_{\overline{a\otimes b\otimes c}}(u\otimes v)\Bigr).
\end{split}
\]

\bigskip

\subsection{}\label{ss:dim7}
In this subsection we will determine the possible isomorphism classes $\calI({\cR})$ for the graded algebra ${\cR}=\Cl\subo(U,q)$ for a seven dimensional vector space $U$. As in Subsection~\textbf{\ref{ss:U}} we take a homogeneous basis
$\{u_1,v_1,\ldots,u_m,v_m,w_1,\ldots,w_{7-2m}\}$, where $0\leq m\leq 3$, with $q(u_i,v_i)=1$ and $\deg(u_i)=g_i=-\deg(v_i)$ for $1\leq i\leq m$, $q(w_j)=1$ and $\deg(w_j)=h_j$, $2h_j=0$, for $1\leq j\leq 7-2m$, with the added condition of $h_1+\dotsm+h_{7-2m}=0$.

Moreover, we will assume that the elements $h_1,\ldots,h_{7-2m}$ are all different. Otherwise, if for instance $h_1=h_2$, substitute $w_1$ and $w_2$ by $u_{m+1}=\tfrac{1}{2}(w_1+\sqrt{-1}w_2)$, $v_{m+1}=\tfrac{1}{2}(w_1-\sqrt{-1}w_2)$, to get a new homogeneous basis where $m$ is increased by $1$.

Let us consider the different possibilities:

\begin{description}
\item[\fbox{$m=3$}] In this case a repeated application of Lemma \ref{le:Uuv} shows that $\Cl\subo(U,q)$ is isomorphic (as a graded algebra) to $\End_\bF(W_1)\otimes\End_\bF(W_2)\otimes\End_\bF(W_3)=\End_\bF(W)$ with $W=W_1\otimes W_2\otimes W_3$ a graded eight-dimensional vector space. Hence $\calI\bigl(\Cl\subo(U,q)\bigr)=[\bF]$ in this case, and note that the $G$-grading on $U$ can be refined to a $\bZ^3$-grading (substituting $g_1$, $g_2$ and $g_3$ by three free generators) with the same properties.

\item[\fbox{$m=2$}] The fact that $h_1$, $h_2$ and $h_3$ are different and sum up to $0$ force that $h_1$ and $h_2$ are nonzero and $h_3=h_1+h_2$. Then Lemma \ref{le:Uuv} shows that $\Cl\subo(U,q)$ is isomorphic, as a graded algebra, to $\End_\bF(W)\otimes \Cl\subo(U',q)$, where $W$ is a graded space of dimension $4$ and $U'$ is the subspace spanned by $w_1,w_2,w_3$. But this latter subalgebra is $\bZ_2^2$-graded ($h_1$ and $h_2$ being the generators of $\bZ_2^2$) with all the homogeneous spaces of dimension $1$, and hence it is a graded division algebra isomorphic to ${\cQ}$. Therefore $\calI\bigl(\Cl\subo(U,q)\bigr)=[{\cQ}]$ here, and note that the grading can be refined to a $\bZ^2\times\bZ_2^2$-grading with the same properties.

\item[\fbox{$m=1$}] Note that the subgroup $H$ generated by the $h_i$'s is a direct sum of copies of $\bZ_2$. The fact that the $h_i$'s are different and its sum is $0$ forces that the rank of $H$, as a vector space over $\bZ_2$, be either $3$ or $4$.

    If the rank is $3$, we may reorder our basis $\{u_1,v_1,w_1,w_2,w_3,w_4,w_5\}$ so that $h_1$, $h_2$ and $h_3$ are $\bZ_2$-linearly independent ($h_i=\deg(w_i)$ for any $i$). Then necessarily, up to the order of $w_4$ and $w_5$, we have $h_4=h_1+h_2+h_3$ and $h_5=0$. By Lemma \ref{le:Uuv}, $\Cl\subo(U,q)$ is isomorphic to $\End_\bF(W)\otimes\Cl\subo(U',q)$, for a two-dimensional graded vector space $W$, where $U'$ is the span of the $w_i$'s. Note that $\Cl\subo(U',q)$ is $\bZ_2^3$-graded (with $h_1,h_2,h_3$ the generators of $\bZ_2^3$). The component $\Cl\subo(U',q)_0$ is spanned by $1$ and $w_1w_2w_3w_4$, and hence $\Cl\subo(U',q)$ is not a graded division algebra, but it is neither isomorphic, as a graded algebra, to $\End_\bF(W')$ for a graded vector space $W'$, since here the homogeneous component of degree $0$ is at least four-dimensional. We conclude that $\calI\bigl(\Cl\subo(U,q)\bigr)=\calI\bigl(\Cl\subo(U',q)\bigr)=[{\cQ}]$ in this case. Again the grading on $U$ can be refined to a $\bZ\times\bZ_2^3$-grading.

    However, if the rank is $4$, then we may reorder the elements of our basis so that $h_1,h_2,h_3,h_4$ are linearly independent. The only possibility left for $h_5$ is $h_5=h_1+h_2+h_3+h_4$. By Lemma \ref{le:Uuv}, $\Cl\subo(U,q)$ is isomorphic to $\End_\bF(W)\otimes\Cl\subo(U',q)$, and $\Cl\subo(U',q)$ is a $\bZ_2^4$-graded division algebra (and thus isomorphic to ${\cQ}\otimes {\cQ}$ by Proposition \ref{pr:graded_division}). We conclude here that $\calI\bigl(\Cl\subo(U,q)\bigr)=\calI\bigl(\Cl\subo(U',q)\bigr)=[{\cQ}\otimes {\cQ}]$.

\smallskip

\item[\fbox{$m=0$}] Here the support of the grading group is generated by the $h_i$'s, and hence it is isomorphic to $\bZ_2^r$ for some $r$. The fact that the $h_i$'s are different, its sum being $0$, forces $3\leq r\leq 6$.

    If $r=6$, $\Cl\subo(U,q)=\alg{zw_1,zw_2,\ldots,zw_6}$ is a graded division algebra, and $\calI\bigl(\Cl\subo(U,q)\bigr)=[{\cQ}\otimes {\cQ}\otimes {\cQ}]$.

    If $r=5$ we may assume that $h_1,\ldots,h_5$ are linearly independent over $\bZ_2$ and we have, up to a reordering of the elements in the basis, two possibilities: either $h_6=h_1+\dotsm+h_5$ and $h_7=0$, or $h_6=h_1+h_2$ and $h_7=h_3+h_4+h_5$. In the first case, the homogeneous subspace of degree $0$ in $\Cl\subo(U,q)$ is spanned by $1$ and $zw_7$, while in the second case it is spanned by $1$ and $w_3w_4w_5w_7$. Arguing as before we conclude that $\calI\bigl(\Cl\subo(U,q)\bigr)=[{\cQ}\otimes {\cQ}]$.

    Assume now that $r=4$. Then, up to a reordering of the basic elements, we can take $h_1,h_2,h_3,h_4$ linearly independent and either $h_5=h_1+h_2$, $h_6=h_3+h_4$, and $h_7=0$, or $h_5=h_1+h_2$, $h_6=h_1+h_3$, and $h_7=h_1+h_4$. In the first case the homogeneous component of degree $0$ in $\Cl\subo(U,q)$ is spanned by $1$, $zw_1w_2w_5$, $zw_3w_4w_6$ and $zw_7$, and hence this degree $0$ component is isomorphic to $\Mat_2(\bF)$, which contains two orthogonal nonzero idempotents whose sum is $1$ but cannot contain four such idempotents. We conclude then that $\calI\bigl(\Cl\subo(U,q)\bigr)=[{\cQ}\otimes {\cQ}]$ holds.
    In the second case we are in the situation of Subsection~\textbf{\ref{ss:Quaternionexample}}, with $h_1=\deg(w_7)$, $h_2=\deg(w_1)$, $h_3=\deg(w_2)$ and $h_4=\deg(w_4)$ in \eqref{eq:wis}, and hence $\calI\bigl(\Cl\subo(U,q)\bigr)=[{\cQ}]$.

    Finally, for $r=3$ and up to a reordering of the basic elements, we can take $h_1,h_2,h_3$ linearly independent, $h_4=h_1+h_2$, $h_5=h_2+h_3$, $h_6=h_1+h_3$ and $h_7=h_1+h_2+h_3$. This is precisely the situation in \textbf{\ref{ss:Cayleyexample}}, and hence $\calI\bigl(\Cl\subo(U,q)\bigr)=[\bF]$.

\end{description}

\bigskip

\section{Some models of the exceptional simple classical Lie superalgebra $F(4)$}\label{se:modelsF(4)}

Let $\frg=\frg\subo\oplus\frg\subuno$ be the simple Lie superalgebra of type $F(4)$. Then $\frg\subo$ is the direct sum of two proper simple ideals:
$\frg\subo=\fra_1\oplus\frb_3$, where $\fra_1$ is simple of type $A_1$ and $\frb_3$ is simple of type $B_3$; while $\frg\subuno$ is, as a module for $\frg\subo$, the tensor product of the two-dimensional irreducible module for $\fra_1$ and the spin eight-dimensional module for $\frb_3$. There is a unique, up to scalars, nonzero $\frg\subo$-invariant map $\frg\subuno\otimes\frg\subuno\rightarrow \frg\subo$ (the bracket of odd elements) satisfying that $\frg$ is a Lie superalgebra (different scalars give isomorphic superalgebras).

\bigskip

\subsection{$F(4)$ and the Cayley algebra}\label{ss:F(4)Cayley}

The ideal $\frb_3$ of $\frg\subo$ can be identified with the orthogonal Lie algebra $\frso(\cC^0,-N)$, and the spin module with $\cC$ for the Cayley algebra $\cC$ as in Subsection~\textbf{\ref{ss:Cayleyexample}}. Then
\[
\frg\subo=\frsp(V)\oplus \frso(\cC^0,-N),\qquad
\frg\subuno=V\otimes\cC,
\]
for a two-dimensional vector space $V$ endowed with a nonzero symplectic form $(u\vert v)$. Take a basis $\{v_1,v_{-1}\}$ of $V$ with $(v_1\vert v_{-1})=1$ and grade $V$ over $\bZ$ with $\deg(v_{\pm 1})=\pm 1$. This grades $\fra_1=\frsp(V)\cong\frsl(V)$ over $\bZ$ with degrees $0,\pm 2$. On the other hand, the $\bZ_2^3$-grading on $\cC$ in \textbf{\ref{ss:Cayleyexample}} is compatible with the norm ($N(\cC_g^0,\cC_h^0)=0$ unless $g+h=0$), so it induces a $\bZ_2^3$-grading on $\frb_3=\frso(\cC^0,-N)$ of type $(0,0,7)$, with the zero component being trivial, and the remaining seven homogeneous components being Cartan subalgebras of $\frso(\cC^0,-N)$.

Combining these gradings we obtain a $\bZ\times\bZ_2^3$-grading on $F(4)$ whose type is $(19,0,7)$.

\bigskip

\subsection{$F(4)$ and the Kac Jordan superalgebra}\label{ss:F(4)Kac}

Given any Jordan superalgebra $\cJ$, its \emph{Tits-Kantor-Koecher Lie superalgebra} is the Lie superalgebra:
\[
\tkk(\cJ)=\Bigl(\cQ^0\otimes\cJ\Bigr)\oplus\der(\cJ),
\]
where $\cQ$ is our quaternion algebra, with bracket determined by $\der(\cJ)$ being a subalgebra, and by:
\[
\begin{split}
[a\otimes x,b\otimes y]&=[a,b]\otimes xy\, - 2N(a,b)[L_x,L_y],\\
[d,a\otimes x]&=a\otimes d(x),
\end{split}
\]
for any $a,b\in\cQ^0$, $x,y\in\cJ$, and $d\in \der(\cJ)$, where $L_x$ denotes the left multiplication by $x$. This construction, in this form, goes back to \cite{Tits62} (in the ungraded case).

\smallskip

The ten-dimensional Kac Jordan superalgebra $\cK_{10}$ is a simple Jordan superalgebra discovered in \cite{KacJordan}, and which is easily described in terms of the smaller Kaplansky superalgebra $\cK_3$ (see \cite{GeorgiaAlberto}). The tiny Kaplansky superalgebra  is the three-dimensional Jordan superalgebra
\[
\cK_3=(\cK_3)\subo\oplus(\cK_3)\subuno
\]
with $(\cK_3)\subo=\bF e$ and $(\cK_3)\subuno=V$, where $V$ is a two-dimensional vector space endowed with a nonzero symplectic form $(u\vert v)$. The multiplication in $\cK_3$ is given by
\[
e^2=e,\qquad ex=xe=\frac12x,\qquad xy=(x\vert y)e,
\]
for any $x,y\in V$. The bilinear form is extended to a supersymmetric bilinear form of $\cK_3$ by means of $(e\vert e)=\frac12$ and $\bigl((\cK_3)\subo\vert
(\cK_3)\subuno\bigr)=0$. Its Lie superalgebra of derivations $\der(\cK_3)$ is just the orthosymplectic Lie superalgebra $\frosp(\cK_3)$ relative to the supersymmetric bilinear form above. In particular $\der(\cK_3)\subo\simeq\frsp(V)$, the symplectic Lie algebra of $V$, acting trivially on $(\cK_3)\subo=\bF e$.

Now in \cite{GeorgiaAlberto} it is shown that the Kac superalgebra $\cK_{10}$ appears as:
\[
\cK_{10}=\bF 1\oplus \bigl(\cK_3\otimes \cK_3\bigr)
\]
with unit element $1$ and with product determined by
\[
(a\otimes b)(c\otimes d)=(-1)^{\bar b\bar c}\left( ac\otimes bd-\frac34 (a\vert c)(b\vert d)1\right)
\]
for homogeneous elements $a,b,c,d\in \cK_3$.

The Lie superalgebra of derivations is $\der(\cK_{10})=(\der(\cK_3)\otimes id) \oplus (id\otimes \der(\cK_3))$ (with trivial action on $1$).

\smallskip

There are two fine gradings on $\cK_{10}$ (see \cite{AntonioCristinaCandidoK10}) which can be described as follows. First take a symplectic basis $\{v_1,v_{-1}\}$ of $V=(\cK_3)\subuno$, so $(v_1\vert v_{-1})=1$. With $\deg(e)=0$, $\deg(v_{\pm 1})=\pm 1$ we obtain a $\bZ$-grading on $\cK_3$ and hence a $\bZ^2$-grading on $\cK_{10}=\bF 1\oplus\bigl(\cK_3\otimes \cK_3\bigr)$, where each copy of $\cK_3$ is graded independently over $\bZ$. That is, $\deg(1)=(0,0)$ and $\deg(x\otimes y)=(\deg(x),\deg(y))$ for homogeneous elements $x,y\in\cK_3$. The type of this grading is $(8,1)$.

On the other hand, there is a natural order $2$ automorphism $\tau$ of $\cK_{10}$ with $\tau(1)=1$ and $\tau(x\otimes y)=(-1)^{\bar x\bar y}y\otimes x$ for homogeneous elements $x,y\in\cK_3$. Then $\cK_{10}$ is $\bZ\times \bZ_2$-graded as follows: first the $\bZ$-grading above on $\cK_3$ determines a $\bZ$-grading on $\cK_{10}$ with $\deg(1)=0$, $\deg(x\otimes y)=\deg(x)+\deg(y)$ for homogeneous elements $x,y\in\cK_3$, and this is refined by means of $\tau$:
\[
(\cK_{10})_{(n,\bar 0)}=\{z\in(\cK_{10})_n: \tau(z)=z\},\quad
(\cK_{10})_{(n,\bar 1)}=\{z\in(\cK_{10})_n: \tau(z)=-z\}.
\]
The homogeneous component of degree $0$ is $\bF 1\oplus\bF (e\otimes e)\oplus \bF(v_1\otimes v_{-1}-v_{-1}\otimes v_1)$, while the dimension of all the other homogeneous components is $1$. Hence the type is $(7,0,1)$.

\smallskip

The even part $(\cK_{10})\subo$ is the direct sum of two simple ideals (of respective dimensions $1$ and $5$):
\[
(\cK_{10})\subo=\bF E_1\oplus\Bigl(\bF E_2\oplus (V\otimes V)\Bigr),
\]
where $E_1=-\frac12 1+2e\otimes e$ and $E_2=\frac32 1-2e\otimes e$ are orthogonal idempotents, and we have $E_2(a\otimes b)=a\otimes b$ and $(a\otimes b)(c\otimes d)=\frac12 (a\vert c)(b\vert d)E_2$ for any $a,b,c,d\in V$. On the other hand, the even part of $\der(\cK_{10})$ is, up to natural identifications,
\[
\der(\cK_{10})\subo=(\der(\cK_3)\subo\otimes id) \oplus (id\otimes \der(\cK_3)\subo)=(\frsp(V)\otimes id) \oplus (id\otimes \frsp(V)),
\]
and note that $(\frsp(V)\otimes id)\oplus (id\otimes \frsp(V))$ is naturally isomorphic to the orthogonal Lie algebra $\frso(V\otimes V,Q)$, where $Q$ is the quadratic form whose associated polar form is given by $Q(a\otimes b,c\otimes d)=(a\vert c)(b\vert d)$ for any $a,b,c,d\in V=(\cK_3)\subuno$.

Extend the quadratic form $Q$ on $V\otimes V$ to the vector space $U=\cQ^0\oplus (V\otimes V)$ by imposing that $\cQ^0$ and $V\otimes V$ are orthogonal and with the restriction of $Q$ to $\cQ^0$ equal to the norm $N$.

\smallskip

The Tits-Kantor-Koecher Lie superalgebra $\tkk(\cK_{10})$ is $F(4)$ (this is how $\cK_{10}$ was discovered in \cite{KacJordan}). Note that
the even part of the Tits-Kantor-Koecher Lie superalgebra $\tkk(\cK_{10})$ is the sum of two orthogonal ideals:
\[
\begin{split}
\tkk(\cK_{10})\subo&=\cQ^0\otimes E_1\,\oplus\Bigl( \cQ^0\otimes \bigl(\bF E_2\oplus (V\otimes V)\bigr)\oplus \der(\cK_{10})\subo\Bigr)\\
 &=\cQ^0\otimes E_1\,\oplus\Bigl( \cQ^0\otimes \bigl(\bF E_2\oplus (V\otimes V)\bigr)\oplus (\frsp(V)\otimes id) \oplus (id\otimes \frsp(V))\Bigr),
\end{split}
\]
the first ideal $\cQ^0\otimes E_1$ being isomorphic to $\cQ^0$ (the simple Lie algebra of type $A_1$). As for the second ideal, the arguments in \cite[pp. 355-356]{Eld.NewSimple} prove the next result:

\begin{lemma}\label{le:tkkK10bar0}
The linear map
\[
\begin{split}
\cQ^0\otimes \bigl(\bF E_2\oplus (V\otimes V)\bigr)\oplus\der(\cK_{10})\subo &\longrightarrow \frso(U,Q)\\
p\otimes E_2\quad&\mapsto\ \ad_p(:q\mapsto [p,q])\in\frso(\cQ^0,N)\\
&\qquad\qquad \text{(which is contained in $\frso(U,Q)$)},\\[2pt]
p\otimes (u\otimes v)\quad&\mapsto\ Q(p,.)u\otimes v-Q(u\otimes v,.)p\in\frso(U,Q),\\[2pt]
d\quad &\mapsto\ d\vert_{V\otimes V}\in\frso(V\otimes V,Q)\subseteq \frso(U,Q),
\end{split}
\]
for $p\in \cQ^0$, $d\in \der(\cK_{10})\subo$ and $u,v\in V=(\cK_3)\subuno$, is a Lie algebra isomorphism.
\end{lemma}

\smallskip

Note then that if the quaternion algebra $\cQ$ is graded over an abelian group $G$ and the Kac superalgebra $\cK_{10}$ is graded over an abelian group $H$, then $\tkk(\cK_{10})$ is naturally $G\times H$ graded, the vector space $U=\cQ^0\oplus (V\otimes V)$ is $G\times H$-graded too in a way compatible with the quadratic form $Q$, and hence it induces a $G\times H$-grading on $\frso(U,Q)$, such that the isomorphism in Lemma \ref{le:tkkK10bar0} is a graded isomorphism.

For the $\bZ_2^2$-grading on $\cQ$ and the $\bZ^2$-grading on $\cK_{10}$, $U$ is $\bZ^2\times \bZ_2^2$-graded and we are in the situation of Subsection~\textbf{\ref{ss:dim7}} with $m=2$. We obtain a $\bZ^2\times \bZ_2^2$-grading on $F(4)=\tkk(\cK_{10})$ of type $(32,4)$.

On the other hand, the $\bZ_2^2$-grading on $\cQ$ combined with the $\bZ\times \bZ_2$-grading on $\cK_{10}$ gives a $\bZ\times\bZ_2^3$-grading on $U$ and we are here in the situation of Subsection~\textbf{\ref{ss:dim7}} with $m=1$ and $\calI\bigl(\Cl\subo(U,Q)\bigr)=[\cQ]$. We obtain a $\bZ\times \bZ_2^3$-grading on $F(4)=\tkk(\cK_{10})$ of type $(31,0,3)$.

\bigskip


\section{Fine gradings on the exceptional simple classical Lie superalgebra $F(4)$}\label{se:gradingsF(4)}

Let $\frg=\frg\subo\oplus\frg\subuno$ be the simple Lie superalgebra of type $F(4)$, and let $\Gamma:\frg=\oplus_{g\in G}\frg_g$ be a fine grading. Let $D=\Diag_\Gamma(\frg)$ be the diagonal group of the grading. As explained in Section~\ref{se:Introduction}, $D$ is the group of automorphisms $\varphi$ of $\frg$ such that for any $g\in G$ there is a nonzero scalar $\alpha_g\in \bF$ such that the restriction of $\varphi$ to $\frg_g$ is the multiplication by the scalar $\alpha_g$. This diagonal group is a maximal abelian diagonalizable  subgroup of $\Aut\frg$. The grading is determined by the action of $D$, the homogeneous components being the common eigenspaces for the elements of $D$.

\subsection{The action of $\frg\subo$ on $\frg\subuno$}\label{ss:g0ong1} Recall that the even part of $\frg$ is the direct sum of its two proper ideals: $\frg\subo=\fra_1\oplus\frb_3$. Since these are the only proper ideals of $\frg\subo$, our grading $\Gamma$ (given by the action of the diagonal group $D$) induces a grading (not necessarily fine) on both $\fra_1$ and $\frb_3$.

Identify $\fra_1$ with $\frsl_2(\bF)$, that is, with the Lie algebra ${\cQ}^0$ of zero trace elements of ${\cQ}=\Mat_2(\bF)$. The well known results on gradings on $\fra_1$ (see \cite{Eld.Fine} and references therein) show that either there is a diagonalizable element $h$ of degree $0$ (that is, fixed by all the elements of $D$), or $\fra_1$ is $\bZ_2^2$-graded (there is a subgroup $H$ isomorphic to $\bZ_2^2$ of the grading group $G$ which grades $\fra_1$).

As a module for $\frg\subo$, the odd part $\frg\subuno$ is the tensor product $\frg\subuno =V\otimes W$ of the two-dimensional irreducible module $V$ for $\fra_1$ and the spin eight-dimensional module $W$ for $\frb_3$. Note that the action of $\fra_1$ on $V$ satisfies $x(xv)=x^2v=-\det(x)v$ for any $x\in
\fra_1\simeq\frsl_2(\bF)$ and any $v\in V$.

But given any representation $\rho:\fra_1\rightarrow \frgl(M)$, $x\mapsto \rho_x$, of our Lie algebra $\fra_1$ such that $\rho_x^2=-\det(x)1$ for any $x\in \fra_1$, the module $M$ becomes a right ${\cQ}$-module (${\cQ}=\Mat_2(\bF)$) as follows:
\begin{equation}\label{eq:rightQmodule}
m(\alpha 1+x)=\alpha m-\rho_x(m),
\end{equation}
for any $\alpha\in \bF$, $x\in \fra_1={\cQ}^0$ and $m\in M$.

This is what happens for $M=\frg\subuno$ as a module for $\fra_1$.

\smallskip

If, in addition, $\fra_1$ is $\bZ_2^2$-graded (grading induced by the $\bZ_2^2$-grading on ${\cQ}$), and $M$ is a graded module for $\fra_1$ (that is, $M=\oplus_{g\in G}M_g$ with $G$ a group containing a copy $H$ of $\bZ_2^2$, and $\rho_{x_h}(m_g)\in M_{h+g}$ for any $h\in H$, $g\in G$, $x_h\in (\fra_1)_h$ and $m_g\in M_g$) then, since ${\cQ}$ is a graded division algebra, $M$ becomes a free right ${\cQ}$-module.

\smallskip

The action of $\frb_3$ on $\frg\subuno$ commutes with the action of $\fra_1$, and hence with the right action of ${\cQ}$. Therefore $\frb_3$ embeds in $\End_{\cQ}(\frg\subuno)$ as a graded Lie subalgebra.

On the other hand, the subalgebra $\frb_3$ is isomorphic to the orthogonal Lie algebra $\frso(U,q)$ for a seven-dimensional vector space $U$ endowed with a nondegenerate quadratic form $q$. The results in \cite{Eld.Fine} show that any grading on $\frso(U,q)$ is determined by a grading on the vector space $U$ (completely determined up to a shifting) compatible with the quadratic form. Hence the grading on the superalgebra $\frg$ determines a unique grading on $U$ like the ones considered in \textbf{\ref{ss:U}}.

The action of $\frb_3$ on $\frg\subuno$ factors through the even Clifford algebra (see \cite{Eld.QuatOctForms})
\begin{equation}\label{eq:iotaPhi}
\frb_3\,\stackrel{\iota}{\hookrightarrow}\,
\Cl\subo(U,q)\xrightarrow{\Phi}\End_{\cQ}(\frg\subuno).
\end{equation}
And by dimension count $\Cl\subo(U,q)$ is isomorphic to $\End_{\cQ}(\frg\subuno)$ through $\Phi$.

Since the gradings on $\frb_3$ and on $\Cl\subo(U,q)$ are determined by the grading on $U$, the embedding $\iota$ preserves these gradings. The fact that $\frg\subuno$ is a graded module for $\frb_3$ and that (the image of) $\frb_3$ generates $\Cl\subo(U,q)$ shows that $\Phi$ is a graded isomorphism.

\bigskip

\subsection{}\label{ss:F(4)toralenA1} We are left with two possibilities according to the grading induced on $\fra_1$. In this subsection we will consider the first case: there exists a diagonalizable element $0\ne h\in(\fra_1)_0$. After scaling if necessary we may assume that, once $\fra_1$ is identified to $\frsl_2(\bF)$, $h$ is the element $\left(\begin{smallmatrix} 1&0\\ 0&-1\end{smallmatrix}\right)$.

The eigenvectors of the action of $h$ in $\frg\subo$ are $2,0,-2$, while those in $\frg\subuno$ are $1,-1$. The eigenspaces are graded subspaces so our fine grading $\Gamma$ is a refinement of the $\bZ$-grading where $\deg(h)=0$, $\deg(e)=2=-\deg(f)$ with $e\leftrightarrow\left(\begin{smallmatrix} 0&1\\0&0\end{smallmatrix}\right)$ and $f\leftrightarrow \left(\begin{smallmatrix} 0&0\\ 1&0\end{smallmatrix}\right)$. The module $W$ for $\frb_3$ can be identified with the graded subspace $\{x\in\frg\subuno: [h,x]=x\}$ and, by restriction, we have a graded isomorphism $\End_{\cQ}(\frg\subuno)\cong\End_\bF(W)$.

The isomorphism $\Phi$ in \eqref{eq:iotaPhi} is a graded isomorphism, so the grading on $\End_\bF(W)$ is completely determined by the grading on $U$, and hence, the grading on $W$, which is the only, up to isomorphism, irreducible graded module for $\End_\bF(W)$, is completely determined, up to a shift, by the grading on $U$. (Note that the grading on $W$ determines the grading on $\frg\subuno$, and hence the grading on the whole $\frg$, as $\frg\subo=[\frg\subuno,\frg\subuno]$.)

Now the fact that $\Phi:\Cl\subo(U,q)\rightarrow\End_{\cQ}(\frg\subuno)\simeq \End_\bF(W)$ is a graded isomorphism shows that $\calI\bigl(\Cl\subo(U,q)\bigr)=[\bF]$ in this case. There are only two possibilities then (see Subsection~\textbf{\ref{ss:dim7}}): either $m=3$ and, since we are looking only for fine gradings, $U$ is graded over $\bZ^3$, or $m=0$ and we have the situation in \textbf{\ref{ss:Cayleyexample}}, where $U$ is graded over $\bZ_2^3$.

\smallskip

In the first case the grading induced on $\frb_3$ is the $\bZ^3$-grading given by the root space decomposition relative to the Cartan subalgebra of $\frb_3$ spanned by the elements $[u_1,v_1]$, $[u_2,v_2]$ and $[u_3,v_3]$ ($\frb_3$ is identified with its image in $\Cl\subo(U,q)$ through $\iota$). Thus the grading we have on $\frg$ is the $\bZ^4$-grading corresponding to the root space decomposition relative to a Cartan subalgebra. The type of this grading is $(36,0,0,1)$.

\smallskip

In the second case, we may identify $U$ with $\cC^0$ and $W$ with $\cC$ for the $\bZ_2^3$-graded Cayley algebra $\cC$ (the spin module for $\frso(\cC^0,-N)$) as in \textbf{\ref{ss:Cayleyexample}} and \textbf{\ref{ss:F(4)Cayley}}. The grading on $W\simeq \cC$ is a shift of the $\bZ_2^3$-grading on $\cC$ and hence we get the same homogeneous components as in the $\bZ\times\bZ_2^3$-grading in Subsection~\textbf{\ref{ss:F(4)Cayley}}. We obtain, up to equivalence, the $\bZ\times\bZ_2^3$-grading in \textbf{\ref{ss:F(4)Cayley}}, of type $(19,0,7)$.
%
%
Note that the grading over $\bZ_2^3$ on $U=\cC^0$ cannot be refined while keeping the condition $\calI\bigl(\Cl\subo(U,q)\bigr)=[\bF]$, and hence the grading induced on $\frg$ is fine.

\bigskip

\subsection{}\label{ss:F(4)notoralenA1} Assume now that our grading on $\frg$ restricts to a $\bZ_2^2$-grading on $\fra_1$, which we identify with ${\cQ}^0$ for the $\bZ_2^2$-graded division algebra ${\cQ}$. Hence $\frg\subuno$ is a free right ${\cQ}$-module of rank $4$, and the map $\Phi:\Cl\subo(U,q)\rightarrow \End_{\cQ}(\frg\subuno)$ in \eqref{eq:iotaPhi} is a graded isomorphism. Thus $\frg\subuno$ is, up to isomorphism, the unique irreducible graded $\Cl\subo(U,q)$-module, its grading being determined, up to a shift, by the grading on $\Cl\subo(U,q)$. Hence $\calI\bigl(\Cl\subo(U,q)\bigr)=[{\cQ}]$, so that we are left, according to \textbf{\ref{ss:dim7}}, with three cases:
\begin{enumerate}
\item $m=2$, and $U$ is graded over $\bZ^2\times\bZ_2^2$ (as we are looking for fine gradings). Again the grading on $U$ uniquely determines, up to a shift, the grading on $\frg\subuno$, and the shift is the one needed to force the product of odd elements $\frg\subuno\times\frg\subuno\rightarrow\frg\subo$, $(x,y)\mapsto [x,y]$, to be a degree $0$ map. Therefore we get a unique possibility, which corresponds to the grading over $\bZ^2\times\bZ_2^2$ in Subsection~\textbf{\ref{ss:F(4)Kac}} on $\tkk(\cK_{10})$ of type $(32,4)$. Again the grading over $\bZ^2\times \bZ_2^2$ on $U=\cQ^0\oplus (V\otimes V)$ cannot be refined while keeping the condition $\calI\bigl(\Cl\subo(U,Q)\bigr)=[\cQ]$, so the grading induced on $\frg$ is fine.

\item $m=1$, and $U$ is graded over $\bZ\times\bZ_2^3$. This corresponds to the grading on $\tkk(\cK_{10})$ obtained by using the $\bZ_2^2$-grading on ${\cQ}^0$ and the fine $\bZ\times \bZ_2$-grading on $\cK_{10}$ as in Subsection~\textbf{\ref{ss:F(4)Kac}}. Its type is $(31,0,3)$ and, as before, it is fine.

\item $m=0$, and we are in the situation of \textbf{\ref{ss:Quaternionexample}}. In particular, $U$ appears as the subspace generated by the elements $w_1,\ldots,w_7$ in \eqref{eq:wis} inside ${\cQ}\otimes {\cQ}\otimes {\cQ}\cong\Cl\subo(U,q)$.

    Since $\Cl\subo(U,q)$ is isomorphic to $\End_{\cQ}(\frg\subuno)$, $\frg\subuno$ is the only, up to isomorphism, graded irreducible module for $\Cl\subo(U,q)$, and thus we may identify $\frg\subuno$ with ${\cQ}\otimes {\cQ}$, with a grading which is a shift of the $\bZ_2^4$-grading on ${\cQ}\otimes {\cQ}$ in \textbf{\ref{ss:Quaternionexample}}.

    The bar involution on $\Cl\subo(U,q)$ gives an orthogonal involution in $\End_{\cQ}(\frg\subuno)$, which is the adjoint involution for the skew-hermitian form $h:({\cQ}\otimes {\cQ})\times({\cQ}\otimes {\cQ})\rightarrow {\cQ}$ in \eqref{eq:h}.

    Since $\frb_3$ is embedded in the subspace of skew elements of $\Cl\subo(U,q)$ relative to the bar involution, the skew-hermitian form $h:\frg\subuno\times\frg\subuno\rightarrow {\cQ}$ is invariant under the action of $\frb_3$. Also, for any $q\in\fra_1={\cQ}^0$ and $x,y\in\frg\subuno={\cQ}\otimes {\cQ}$
    \[
    \begin{split}
    h\bigl([q,x],y\bigr)+h\bigl(x,[q,y]\bigr)&=-h(xq,y)-h(x,yq)\\
    &=-\bar qh(x,y)-h(x,y)q\\
    &=[q,h(x,y)].
    \end{split}
    \]
    (Recall that $\frg\subuno$ becomes a right ${\cQ}$-module with $[q,x]=-xq$ for any $x\in \frg\subuno$ and $q\in {\cQ}^0$. The arguments here are inspired in \cite{Eld.QuatOctForms}.)

    Thus, relative to the adjoint action of ${\cQ}^0$ on ${\cQ}$, $h:\frg\subuno\times\frg\subuno\rightarrow {\cQ}$ is $\fra_1$-invariant too.

    Therefore the bilinear map
    \[
    \begin{split}
    \frg\subuno\times\frg\subuno&\rightarrow {\cQ}^0\\
    (x,y)\, &\mapsto h(x,y)+h(y,x)=h(x,y)-\overline{h(x,y)},
    \end{split}
    \]
    is invariant under the action of both $\fra_1$ and $\frb_3$, and hence under the action of $\frg\subo$. But $\Hom_{\frg\subo}(\frg\subuno\otimes\frg\subuno,\fra_1)$ is one-dimensional and generated by the projection onto $\fra_1$ of the Lie bracket of elements in $\frg\subuno$, which is a degree $0$ map.

    Hence $\frg\subuno={\cQ}\otimes {\cQ}$ is graded over a group $G$ containing a copy of $\bZ_2^4$ in such a way that its grading is a shift by certain $g\in G$ of the grading on ${\cQ}\otimes {\cQ}$ in \textbf{\ref{ss:Quaternionexample}} so that the skew-hermitian form $h$ above becomes a degree $0$ map. That is, for homogeneous elements $x,y$ in ${\cQ}$:
    \[
    \deg_{\frg\subuno}(x\otimes y)=g+\deg_{{\cQ}\otimes {\cQ}}(x\otimes y)=g+(\deg(x),\deg(y)).
    \]
    On the other hand, for homogeneous $x\otimes y\in (\frg\subuno)_a$ and $u\otimes v\in(\frg\subuno)_b$, $h(x\otimes y,u\otimes v)=\bar yq_2vN(x,u)$ must belong to ${\cQ}_{a+b}$. It follows then that $2g=(0,\deg(q_2))$. (Note that ${\cQ}$ is graded over $0\times\bZ_2^2\leq \bZ_2^4$.)

    Therefore $\frg\subuno={\cQ}\otimes {\cQ}$ is graded over a group isomorphic to $\bZ_2^3\times \bZ_4$, which contains the copy of $\bZ_2^4$ given by $\bZ_2^3\times (2\bZ_4)$, with $g=(\bar 0,\bar 0,\bar 0,\bar 1)\in\bZ_2^3\times \bZ_4$, and
    \[
    \widetilde{\deg}(x\otimes y)=\bigl(\deg(x),\deg'(y)\bigr)\in\bZ_2^2\times (\bZ_2\times\bZ_4)=\bZ_2^3\times\bZ_4,
    \]
    where we have
    \[
    \deg'(1)=(\bar 0,\bar 1),\ \deg'(q_1)=(\bar 1,\bar 1),\ \deg'(q_2)=(\bar 0,\bar 3),\ \deg'(q_3)=(\bar 1,\bar 3)
    \]
    in $\bZ_2\times\bZ_4$.

    We conclude that $\frg$ is graded over $\bZ_2^3\times\bZ_4$. The type of the grading is $(16)$ on $\frg\subuno$, $(3)$ on $\fra_1$, and $(6,6,1)$ on $\frb_3$. This type on $\frb_3$ is computed by taking into account the degrees of the elements $w_i$'s in \eqref{eq:wis} (considered as degrees on $\bZ_2^3\times 2\bZ_4$) and checking the degrees of the homogeneous basis $\{[w_i,w_j]: 1\leq i<j\leq 7\}$ of $\frb_3$ (identified with its image in $\Cl\subo(U,q)$). The homogeneous subspace of dimension $3$ corresponds to the degree $(\bar 0,\bar 0,\bar 1,\bar 3)$ (the degree of $w_7$), which is one of the degrees appearing in $\fra_1$. Hence the overall type of the grading on $\frg$ is $(24,6,0,1)$.

    Note that this grading on $\frg$ is fine too, as any refinement would give a refinement in $\frb_3$, which would come from a refinement of the grading on $U$, but any proper refinement of this grading on $U$ gives a value of $\calI\bigl(\Cl\subo(U,q)\bigr)$ different from $[{\cQ}]$.

\end{enumerate}

The next result summarizes our results on gradings on $F(4)$:

\begin{theorem}\label{th:MainThF(4)}
The fine gradings on the simple Lie superalgebra $F(4)$ are, up to equivalence, the following:
\begin{romanenumerate}
\item The Cartan grading over $\bZ^4$ of type $(36,0,0,1)$.
\item A $\bZ\times\bZ_2^3$-grading of type $(19,0,7)$.
\item A $\bZ^2\times\bZ_2^2$-grading of type $(32,4)$.
\item A $\bZ\times\bZ_2^3$-grading of type $(31,0,3)$.
\item A $\bZ_4\times\bZ_2^3$-grading of type $(24,6,0,1)$.
\end{romanenumerate}
\end{theorem}

\bigskip


\section{Fine gradings on the exceptional simple classical Lie superalgebra $G(3)$}\label{se:gradingsG(3)}

Let $\frg=\frg\subo\oplus\frg\subuno$ be the simple Lie superalgebra of type $G(3)$, and let $\Gamma:\frg=\oplus_{g\in G}\frg_g$ be a fine grading, $\frg_g=(\frg\subo)_g\oplus(\frg\subuno)_g$. Let $D=\Diag_\Gamma(\frg)$ be the diagonal group of the grading.

\smallskip

The even part of $\frg$ is the direct sum of its two proper ideals: $\frg\subo=\fra_1\oplus\frg_2$, where $\fra_1$ is simple of type $A_1$ and $\frg_2$ is simple of type $G_2$. Since these are the only proper ideals of $\frg\subo$, our grading $\Gamma$ (that is, the action of the diagonal group $D$) induces a grading (not necessarily fine) on both $\fra_1$ and $\frg_2$.

Identify again $\fra_1$ with $\frsl_2(\bF)$. Then either there is a diagonalizable element $h\in\frsl_2(\bF)$ of degree $0$ (that is, fixed by all the elements of $D$), or $\fra_1$ is $\bZ_2^2$-graded.

As a module for $\frg\subo$, the odd part $\frg\subuno$ is the tensor product  of the two-dimensional irreducible module for $\fra_1$ and the unique seven-dimensional irreducible module for $\frg_2$.

As in the $F(4)$ case, if $\fra_1$ is $\bZ_2^2$-graded (grading induced by the $\bZ_2^2$-grading on ${\cQ}$), then $\frg\subuno$ becomes a free right ${\cQ}$-module. But the dimension of $\frg\subuno$ is $14$, which is not a multiple of $4$. Hence this case is impossible.

\smallskip

Therefore our fine grading is a refinement of the $\bZ$-grading given by the eigenspaces of the operator $\ad_h$, with $h$ above. On the other hand, our grading induces a grading on $\frg_2$, and this is induced by a grading on its unique irreducible seven-dimensional module, considered as the (Malcev) algebra of zero trace elements in a Cayley-Dickson algebra (see \cite{CristinaCandidoG2}, \cite{BahturinTvalavadzeG2} or \cite{EKG2F4}). There are just two non-equivalent such fine gradings \cite{EldGrOctonions}: over $\bZ^2$ and over $\bZ_2^3$, and this gives the two fine gradings on $\frg$, with grading groups $\bZ^3$ and $\bZ\times\bZ_2^3$, and respective types $(28,0,1)$ and $(17,7)$.

\smallskip

Therefore, we get:

\begin{theorem}\label{th:MainThG(3)}
The fine gradings on the simple Lie superalgebra $G(3)$ are, up to equivalence, the following:
\begin{romanenumerate}
\item The Cartan grading over $\bZ^3$ of type $(28,0,1)$.
\item A $\bZ\times\bZ_2^3$-grading of type $(17,7)$.
\end{romanenumerate}
\end{theorem}

\bigskip


\section{The automorphism group of $D(2,1;\alpha)$}\label{se:AutD21a}

\subsection{}\label{ss:notationD21a} Let $V_i$, $i=1,2,3$, be a two-dimensional vector space endowed with a nonzero skew-symmetric bilinear form $b_i:V_i\times V_i\rightarrow \bF$. Denote by $\frsp(V_i)$ the corresponding symplectic Lie algebra, which coincides with the Lie algebra of trace $0$ endomorphisms $\frsl(V_i)$.

Then the exceptional simple Lie superalgebra $\frg=D(2,1;\alpha)$, $\alpha\ne 0,-1$, is the Lie superalgebra with even and odd parts given by:
\begin{equation}\label{eq:evenoddD21a}
\begin{split}
\frg\subo&=\frsp(V_1)\oplus\frsp(V_2)\oplus\frsp(V_3),\\
\frg\subuno&=V_1\otimes V_2\otimes V_3,
\end{split}
\end{equation}
with the natural Lie bracket on $\frg\subo$, the natural action of $\frg\subo$ on $\frg\subuno$, and with the product of odd elements given by:
\begin{equation}\label{eq:oddproductD21a}
\begin{split}
[u_1\otimes u_2\otimes u_3,v_1\otimes v_2\otimes v_3]
&=b_2(u_2,v_2)b_3(u_3,v_3)\gamma^1_{u_1,v_1}\\
&\quad +\alpha b_1(u_1,v_1)b_3(u_3,v_3)\gamma^2_{u_2,v_2}\\
&\quad+
(-1-\alpha)b_1(u_1,v_1)b_2(u_2,v_2)\gamma^3_{u_3,v_3}
\end{split}
\end{equation}
for any $u_i,v_i\in V_i$, $i=1,2,3$, where $\gamma^i_{u_i,v_i}=b_i(u_i,.)v_i+b(v_i,.)u_i\,(\in\frsp(V_i))$. This algebra is denoted by $\Gamma(1,\alpha,-1-\alpha)$ in \cite{Scheunert}.

It follows that $D(2,1;\alpha)$ is isomorphic to $D(2,1;\beta)$ ($\alpha,\beta\ne 0,-1$) if and only if
\[
\beta\in\left\{\alpha,-1-\alpha,\frac{1}{\alpha},\frac{-1}{1+\alpha},
\frac{-\alpha}{1+\alpha},\frac{1+\alpha}{-\alpha}\right\}.
\]
These isomorphisms are obtained by permuting the factors of $\frg\subuno$ (and the corresponding simple ideals in $\frg\subo$). Thus, for instance, change $b_1$ to $\alpha b_1$, then Equation~\eqref{eq:oddproductD21a} appears as $[u_1\otimes u_2\otimes u_3,v_1\otimes v_2\otimes v_3]
=\frac{1}{\alpha}b_2(u_2,v_2)b_3(u_3,v_3)\gamma^1_{u_1,v_1}+ b_1(u_1,v_1)b_3(u_3,v_3)\gamma^2_{u_2,v_2}+
\frac{-1-\alpha}{\alpha}b_1(u_1,v_1)b_2(u_2,v_2)\gamma^3_{u_3,v_3}$. Now permute the two first factors to get $D(2,1;\alpha)\cong D\bigl(2,1;\frac{1}{\alpha}\bigr)$.

\smallskip

\subsection{}\label{ss:AutD21a} For any $f_i$ in the symplectic group $Sp(V_i)$, $i=1,2,3$, the map $\iota_{(f_1,f_2,f_3)}\in \End(\frg)$ given by:
\begin{equation}\label{eq:iotafis}
\begin{split}
\iota_{(f_1,f_2,f_3)}\bigl((x_1,x_2,x_3)\bigr)&=\bigl(f_1x_1f_1^{-1},f_2x_2f_2^{-1},f_3x_3f_3^{-1}\bigr),\\
\iota_{(f_1,f_2,f_3)}(u_1\otimes u_2\otimes u_3)&=f_1(u_1)\otimes f_2(u_2)\otimes f_3(u_3),
\end{split}
\end{equation}
for $x_i\in \frsp(V_i)$, $u_i\in V_i$, $i=1,2,3$, is an automorphism of the Lie superalgebra $\frg=D(2,1;\alpha)$.

Moreover, for $f_i,g_i\in Sp(V_i)$, $i=1,2,3$, $\iota_{(f_1,f_2,f_3)}=\iota_{(g_1,g_2,g_3)}$ if and only if $g_i=\epsilon_i f_i$ with $\epsilon_i=\pm 1$, $i=1,2,3$, and $\epsilon_1\epsilon_2\epsilon_3=1$.

Consider $K=\{(\epsilon_1I_{V_1},\epsilon_2I_{V_2},\epsilon_3I_{V_3}):\epsilon_i=\pm 1,\, i=1,2,3,\ \epsilon_1\epsilon_2\epsilon_3=1\}$, where $I_{V_i}$ denotes the identity map on $V_i$, then we get a one-to-one group homomorphism:
\begin{equation}\label{eq:SpViKAutD21a}
\begin{split}
\bigl(Sp(V_1)\times Sp(V_2)\times Sp(V_3)\bigr)/K&\rightarrow \Aut\frg\\
(f_1,f_2,f_3)K&\mapsto \iota_{(f_1,f_2,f_3)},
\end{split}
\end{equation}
and this is an isomorphism if $\alpha\ne 1,-\tfrac{1}{2},-2$ and $\alpha$ is not a primitive cubic root of $1$ (see \cite{Serganova}).

Let us denote by $\Int\frg$ the image of the homomorphism in \eqref{eq:SpViKAutD21a}.

\smallskip

\subsection{}\label{ss:osp42} For $\alpha\in\{1,-\tfrac{1}{2},-2\}$, the Lie superalgebra $\frg=D(2,1;\alpha)$ is isomorphic to the orthosymplectic Lie superalgebra $\frosp(4,2)$.

For $\alpha=-\tfrac{1}{2}$ we get the order $2$ automorphism $\pi_{2,3}$ which ``permutes'' $V_2$ and $V_3$ in $\frg\subuno$ (once these spaces are identified by means of a linear isomorphism which takes $b_2$ to $b_3$) and $\frsp(V_2)$ and $\frsp(V_3)$ in $\frg\subo$, and it turns out that the automorphism group $\Aut\frg$ is the semidirect product of $\Int\frg$ and the order two cyclic subgroup generated by $\pi_{2,3}$. (Note that if $\alpha=1$ we have to use $\pi_{1,2}$, while if $\alpha=-2$, it is $\pi_{1,3}$ the automorphism to be used.)

\smallskip

Again for $\alpha=-\tfrac{1}{2}$, one may consider the four-dimensional vector space $W=V_2\otimes V_3$, which is endowed with the nondegenerate symmetric bilinear form $b:W\times W\rightarrow \bF$ given by $b\bigl(u_2\otimes u_3,v_2\otimes v_3)=b_2(u_2,v_2)b_3(u_3,v_3)$ for any $u_2,v_2\in V_2$ and $u_3,v_3\in V_3$. Then with $V=V_1$, the odd part is $\frg\subuno=V\otimes W$.

Then the orthogonal group $O(W)$ is the semidirect product of the subgroup $\{\iota^W_{(f_2,f_3)}:f_2\in Sp(V_2),f_3\in Sp(V_3)\}$ (where $\iota^W_{(f_2,f_3)}(u_2\otimes u_3)=f_2(u_2)\otimes f_3(u_3)$) and the order two cyclic group generated by the permutation of the two factors $V_2$ and $V_3$ as before. Besides, the orthogonal Lie algebra $\frso(W)$ is naturally isomorphic to $\frsp(V_2)\oplus\frsp(V_3)$, so that we may identify $\frg\subo$ with $\frsp(V)\oplus\frso(W)$.

It turns out that
\begin{equation}\label{eq:Autosp42}
\Aut D(2,1;-\tfrac{1}{2})=\left\{\tilde\iota_{(f,g)}:f\in Sp(V),g\in O(W)\right\},
\end{equation}
where
\begin{equation}\label{eq:iotatilde}
\begin{split}
\tilde\iota_{(f,g)}(u\otimes w)&=f(u)\otimes g(w),\\ \tilde\iota_{(f,g)}(x,y)&=(fxf^{-1},gyg^{-1}),
\end{split}
\end{equation}
for any $u\in V$, $w\in W$, $x\in\frsp(V)$ and $y\in\frso(W)$. Thus $\Aut D(2,1;-\tfrac{1}{2})$ is isomorphic to $Sp(V)\times O(W)/\tilde K$, where $\tilde K=\{\pm(I_V,I_W)\}$.

\medskip

\subsection{}\label{ss:omega} Finally, if $\alpha$ is a primitive cubic root $\omega$ of $1$, then $-1-\alpha=\omega^2$. Identify $(V_1,b_1)$, $(V_2,b_2)$ and $(V_3,b_3)$ by means of suitable linear maps, and then $\Aut\frg$ is the semidirect product of $\Int\frg$ and the order three cyclic group generated by the automorphism $\varpi$ of $\frg=D(2,1;\omega)$ given by:
\begin{equation}\label{eq:varpi}
\begin{split}
\varpi(x_1,x_2,x_3)&=(x_3,x_1,x_2),\\
\varpi(u_1\otimes u_2\otimes u_3)&=\omega(u_3\otimes u_1\otimes u_2),
\end{split}
\end{equation}
for $x_i\in\frsp(V_i)$ and $u_i\in V_i$, $i=1,2,3$.

\medskip

\subsection{}\label{ss:quaterniongroup} Take a symplectic basis $\{u_l,v_l\}$ of $V_l$, $l=1,2,3$ (that is, $b_l(u_l,v_l)=1$). Inside each symplectic group $Sp(V_l)$ there appears the subgroup generated by the order $4$ elements $a_l$ and $b_l$ whose coordinate matrices in the basis $\{u_l,v_l\}$ are:
\begin{equation}\label{eq:ajbj}
a_l\leftrightarrow \begin{pmatrix} i&0\\ 0&-i\end{pmatrix},\qquad
b_l\leftrightarrow \begin{pmatrix} 0&-1\\ 1&0\end{pmatrix}.
\end{equation}
(Here $i$ denotes a square root of $-1$.) These two elements generate an order $8$ subgroup which is isomorphic to the quaternion group $Q_8=\{\pm 1,\pm i,\pm j,\pm k\}$ (a subgroup of the multiplicative group of the real division algebra of quaternions $\bH$). Here $i^2=j^2=-1$, $ij=-ji=k$.

\begin{remark}\label{re:Q8}
The elements $a_l,b_l$ in $Sp(V_l)$ satisfy the relations $a_l^4=b_l^4=1$ and $a_lb_l=-b_la_l$. It is an easy exercise to note that given any two elements in $Sp(V_l)$ satisfying these properties, there is a symplectic basis of $V_l$ such that the coordinate matrices of these two elements are the ones above for $a_l$ and $b_l$. This shows in particular that any automorphism of $Q_8$ is realized by means of a symplectic automorphism of $V_l$.
\end{remark}

\smallskip

Then, for arbitrary $\alpha\ne 0,-1$, inside our group $\Int\frg$ ($\frg=D(2,1;\alpha)$), there appears a subgroup isomorphic to $(Q_8)^3/K$, with $K=\{(\epsilon_1,\epsilon_2,\epsilon_3)\in\{\pm 1\}^3:  \epsilon_1\epsilon_2\epsilon_3=1\}$. This subgroup will play a key role later on.

We will need the following result:

\begin{proposition}\label{pr:Q83K}
Let $G$ be the group $(Q_8)^3/K$ as above, and let $H$ be a maximal abelian subgroup of $G$. Then $H$ is isomorphic to $\bZ_2^2\times \bZ_4$ and, up to conjugation by elements in the subgroup of $\Aut G$ generated by triples of automorphisms of $Q_8$ and permutations of the three copies of $Q_8$, $H$ is one of the following:
\begin{itemize}
\item $\langle i\rangle^3/K$,
\item $\{(x,x,y)\in (Q_8)^3: x\in Q_8,\, y\in \langle i\rangle\}K/K$,
\item $\langle (i,i,i)K,(j,j,i)K,(i,j,j)K\rangle/K$.
\end{itemize}
\end{proposition}

(Here $\langle S\rangle$ denotes the subgroup generated by $S$.)

In order to prove this result we need a previous lemma:

\begin{lemma}\label{le:A1A2A3}
Let $\bF_2$ be the field of two elements, and let $A_1,A_2,A_3$ be three vector spaces over $\bF_2$ of dimension $2$. Let $B$ be a vector subspace of $A_1\oplus A_2\oplus A_3$ such that for any $b_1,b_2\in B$ the number of indices $i=1,2,3$ such that the projection of $\bF_2 b_1+\bF_2b_2$ onto $A_i$ is surjective is even, and $B$ is maximal with these properties. Then, up to a reordering of the indices, we have one of these possibilities:
\begin{enumerate}
\item $\dim B\cap A_i=1$ for any $i=1,2,3$ and $B=(B\cap A_1)\oplus(B\cap A_2)\oplus (B\cap A_3)$.
\item There is a linear isomorphism $f:A_1\rightarrow A_2$ and an element $0\ne a_3\in A_3$ such that $B=\{a+f(a):a\in A_1\}\oplus\bF_2 a_3$.
\item There are linear isomorphisms $f^i:A_1\rightarrow A_i$, $i=2,3$, and a basis $\{a,b\}$ of $A_1$, such that
    $B=\bF_2(a+f^2(a)+f^3(a))\oplus\bF_2(b+f^2(b))\oplus\bF_2(b+f^3(b))$.
\end{enumerate}
\end{lemma}
\begin{proof}
It is easy to check that the vector subspaces $B$ in these three cases satisfy the hypotheses of the lemma.

Conversely, if $B$ is a vector subspace satisfying the hypotheses of the lemma and $B\cap A_i\ne 0$ for any $i=1,2,3$, then we are in case 1).

Thus we may assume $B\cap A_1 =0$. Also the conditions on $B$ force $\dim B\cap(A_2\oplus A_3)\leq 2$. We split into different subcases:
\begin{itemize}
\item  If $B\cap(A_2\oplus A_3)=0$, then $\dim B\leq 2$. If $\dim B\leq 1$, then $B$ is contained in a larger subspace as in item 1), thus contradicting its maximality. If $\dim B=2$, there is a basis $\{a_1,b_1\}$ of $A_1$, and elements $a_2,b_2\in A_2$ and $a_3,b_3\in A_3$ such that $B=\bF_2(a_1+a_2+a_3)\oplus\bF_2(b_1+b_2+b_3)$. Now if $a_2$ and $b_2$ are linearly independent but $a_3$ and $b_3$ are not, then $B$ is contained in $\bF_2(a_1+a_2)\oplus\bF_2(b_1+b_2)\oplus\bF_2 c$ for some $0\ne c\in A_3$, which is a subspace as in item 2). The same happens (after permuting the indices $2$ and $3$) if $a_2$ and $b_2$ are linearly dependent, but $a_3$ and $b_3$ are linearly independent. The possibility of both $a_2$ and $b_2$ and $a_3$ and $b_3$ being linearly independent is ruled out by the hypotheses on $B$.
\item If $B\cap(A_2\oplus A_3)$ has dimension $1$, there is a nonzero element $c_2\in A_2$ and an element $c_3\in A_3$ such that $c_2+c_3\in B$. Here if the dimension of $B$ is $1$, again $B$ is contained in a subspace as in item 1). If $\dim B=2$, then there are elements $0\ne a_1\in A_1$, $a_2\in A_2$ and $a_3\in A_3$ such that $B=\bF_2(a_1+a_2+a_3)\oplus\bF_2(c_2+c_3)$. The hypotheses on $B$ show that either $a_2\in \bF_2c_2$ and $\dim \bF_2a_3+\bF_2c_3\leq 1$, in which case $B$ is contained again in a subspace as in item 1), or $a_2$ and $c_2$ are linearly independent in $A_2$ and so are $a_3$ and $c_3$ in $A_3$. In the latter case $B$ is contained in the subspace $\bF_2a_1\oplus\bF_2(a_2+a_3)\oplus\bF_2(c_2+c_3)$ which, after a cyclic permutation of the indices, is a subspace as in item 2).\\
    Finally, if $\dim B=3$, then there are elements $a_i,b_i\in A_i$, $i=1,2,3$ such that $B=\bF_2(a_1+a_2+a_3)\oplus\bF_2(b_1+b_2+b_3)\oplus\bF_2(c_2+c_3)$, with $a_1$ and $b_1$ linearly independent. We are left with several subcases:
    \begin{itemize}
    \item If for instance $a_2$ and $c_2$ are linearly independent, we may add to the vector $b_1+b_2+b_3$ a suitable linear combination of the other two basic vectors to get the same situation but with $b_2=0$. Then the hypotheses on $B$ force $b_3$ and $c_3$ to be linearly dependent, but both $a_3$ and $b_3$, and $a_3$ and $c_3$ must be linearly independent. Hence $b_3=c_3$ and, up to a permutation of the indices $1$ and $3$, we are in the situation of item 3). The same happens if $a_3$ and $c_3$ are linearly independent, or $b_2$ and $c_2$ are linearly independent, or $b_3$ and $c_3$ are linearly independent.
    \item Hence we may assume that the pairs $\{a_2,c_2\}$, $\{a_3,c_3\}$, $\{b_2,c_2\}$, and $\{b_3,c_3\}$ are linearly dependent. Then since $c_2\ne 0$, we may assume $a_2=b_2=0$: $B=\bF_2(a_1+a_3)\oplus\bF_2(b_1+b_3)\oplus\bF_2(c_2+c_3)$. Hence our hypotheses imply that $a_3$ and $b_3$ are linearly independent and hence $c_3=0$ (as $c_3$ is linearly dependent with both $a_3$ and $b_3$). Up to a permutation of the indices $2$ and $3$ we are in the situation of item 2).
    \end{itemize}
\item Finally, if $\dim B\cap(A_2\oplus A_3)=2$, then there are elements $a_2+a_3,b_2+b_3\in B$ with $a_2$ and $b_2$ linearly independent in $A_2$. Our hypotheses show that then $a_3$ and $b_3$ are linearly independent in $A_3$. Assume that $x_1+x_2+x_3\in B$ with $x_1\ne 0$. Then by adding a suitable linear combination of $a_2+a_3$ and $b_2+b_3$ we may assume that $x_2=0$, and hence both $x_3$ and $a_3$, and $x_3$ and $b_3$ must be linearly dependent. This forces $x_3=0$. Up to a permutation of indices $1$ and $3$ we are in the situation of item 2).
\end{itemize}
\end{proof}

\smallskip

Let us proceed now to the proof of Proposition \ref{pr:Q83K}.
Note first that the center $Z(G)$ of our group $G=(Q_8)^3/K$ is the subgroup $\{(\epsilon_1,\epsilon_2,\epsilon_3):\epsilon_l=\pm 1,l=1,2,3\}/K$ which is a subgroup of order $2$, and the quotient $G/Z(G)$ is isomorphic to $\Bigl(Q_8/Z(Q_8)\Bigr)^3$ which, in turn, is isomorphic to the abelian group $\bZ_2^6$, since $Q_8/Z(Q_8)$ is the direct product of two copies of the cyclic group of order two. Therefore, $G/Z(G)$ is a vector space over $\bF_2$.

Consider the elements $e_1=(i,1,1)K$, $e_2=(j,1,1)K$, $e_3=(1,i,1)K$, $e_4=(1,j,1)K$, $e_5=(1,1,i)K$, $e_6=(1,1,j)K$, and $-1=(-1,-1,-1)K$. Then $e_l^2=-1$ for any $l=1,\ldots,6$, and in $G$ we have $e_re_s=e_se_r$ unless $\{r,s\}=\{2l-1,2l\}$ ($l=1,2,3$), in which case we have $e_re_s=(-1)e_se_r$.
Denote by $\bar e_l$ the class of $e_l$ modulo the center. Then, as a vector space over $\bF_2$, $G/Z(G)=A_1\oplus A_2\oplus A_3$ with $A_1 =\bF_2 \bar e_1+\bF_2\bar e_2$, $A_2 =\bF_2 \bar e_3+\bF_2\bar e_4$, and $A_3 =\bF_2 \bar e_5+\bF_2\bar e_6$. Moreover, any maximal abelian subgroup of $G$ contains the center $Z(G)$ and its quotient modulo this center is a subspace of $G/Z(G)$ satisfying the hypotheses of the previous lemma. Whence the result. \qed

\medskip

In the same vein, inside the symplectic group $Sp(V_1)$ there appears a subgroup isomorphic to the torus $\bF^\times$. This is just the subgroup consisting of the endomorphisms whose coordinate matrix in the symplectic basis $\{u_1,v_1\}$ is $\spmatrix{\mu}$. Then inside our group $\Int\frg$ ($\frg=D(2,1;\alpha)$), there appears also a  subgroup isomorphic to $(\bF^\times\times Q_8\times Q_8)/K$, with $K=\{(\epsilon_1,\epsilon_2,\epsilon_3)\in\{\pm 1\}^3:  \epsilon_1\epsilon_2\epsilon_3=1\}$.

With a simpler proof than the one for Lemma \ref{le:A1A2A3} we obtain:

\begin{lemma}\label{le:A1A2}
Let $\bF_2$ be the field of two elements, and let $A_1,A_2$ be two vector spaces over $\bF_2$ of dimension $2$. Let $B$ be a vector subspace of $A_1\oplus A_2$ such that for any $b_1,b_2\in B$ the number of indices $i=1,2$ such that the projection of $\bF_2 b_1+\bF_2b_2$ onto $A_i$ is surjective is even. Then we have one of these possibilities:
\begin{enumerate}
\item There are nonzero elements $a_1\in A_1$ and $a_2\in A_2$ such that $B$ is contained in $\bF_2a_1\oplus\bF_2a_2$.
\item There is a linear isomorphism $f:A_1\rightarrow A_2$  such that $B=\{a+f(a):a\in A_1\}$.
\end{enumerate}
\end{lemma}

This lemma makes easy the proof of the next result, which will be used later on:

\begin{proposition}\label{pr:FxQ82K}
Let $G$ be the group $(\bF^\times\times Q_8\times Q_8)/K$ as above, and let $H$ be a maximal abelian subgroup of $G$. Then, up to conjugation by elements in the subgroup of $\Aut G$ generated by couples of automorphisms of $Q_8$, $H$ is one of the following:
\begin{itemize}
\item $\left(\bF^\times\times\langle i\rangle\times\langle i\rangle\right)/K$,
\item $\left(\bF^\times\times \{(x,x):x\in Q_8\}\right)K/K$.
\end{itemize}
\end{proposition}
\begin{proof}
It must be noted first that the subgroup $(\bF^\times\times \{\pm 1\}\times\{\pm 1\})/K$ is precisely the center of $G$, so it is contained in any maximal abelian subgroup. Moreover, the quotient $G/Z(G)$ is isomorphic to $\left(Q_8/Z(Q_8)\right)^2$, which is isomorphic to the abelian group $\bZ_2^4$. Now $HZ(G)/Z(G)$ is a vector subspace of $G/Z(G)$, considered as a vector space over $\bF_2$, and now Lemma \ref{le:A1A2} applies.
\end{proof}

\bigskip


\section{Inner fine gradings on $D(2,1;\alpha)$}\label{se:innergradingsD21a}

Let $\frg=\frg\subo\oplus\frg\subuno$ be the simple Lie superalgebra of type $D(2,1;\alpha)$, $\alpha\ne 0,-1$, and let $\Gamma:\frg=\oplus_{g\in G}\frg_g$ be a fine grading, $\frg_g=(\frg\subo)_g\oplus(\frg\subuno)_g$. Let $D=\Diag_\Gamma(\frg)$ be the diagonal group of the grading, which is a maximal abelian diagonalizable subgroup of $\Aut \frg$.

In this section we will deal with the case in which $D$ is contained in $\Int\frg\simeq \left(Sp(V_1)\times Sp(V_2)\times Sp(V_3)\right)/K$. In this case, $D$ fixes each simple ideal $\frsp(V_l)$ ($l=1,2,3$) of $\frg\subo$, and hence it induces a grading on each such ideal.

\smallskip

\subsection{}\label{ss:Gamma123Z22} Assume that the grading induced on each $\frsp(V_l)$ is a fine $\bZ_2^2$-grading. A symplectic basis $\{u_l,v_l\}$ can then be chosen on each $V_l$, so that $D$ is contained in the subgroup of $\Int\frg$ isomorphic to $(Q_8)^3/K$ considered in the previous section.

Proposition \ref{pr:Q83K} gives the possibilities here for $D$, but only the last one satisfies that the grading on each $\frsp(V_l)$ is a fine $\bZ_2^2$-grading for each $l=1,2,3$. Therefore, up to a change of the symplectic basis on each $V_l$ (see Remark \ref{re:Q8}) it turns out that $D$ is the group generated by the elements:
\[
\iota_{(a_1,a_2,a_3)},\ \iota_{(b_1,b_2,a_3)},\ \iota_{(a_1,b_2,b_3)},
\]
which is isomorphic to $\bZ_2^2\times \bZ_4$.

The homogeneous spaces are the common eigenspaces for these three generators. The type of this grading is easily checked to be $(14,0,1)$, with all the homogeneous components in $\frg\subuno$ of dimension $1$. This grading is indeed fine, no matter what the value of $\alpha$ is. Actually, $D$ is indeed a maximal abelian subgroup of $\Aut\frg$ because its centralizer is easily proved to be contained in the subgroup of $\Int\frg$ isomorphic to $(Q_8)^3/K$ considered so far, and $D$ is a maximal abelian subgroup here.

\medskip

\subsection{}\label{ss:Gamma123Z} If for an index $l=1,2,3$, the restriction of $D$ to $\frsp(V_l)$ does not induce a $\bZ_2^2$-grading, then we know that there is a diagonalizable element $h\in\frsp(V_l)$ in the zero component. We may take $l$ to be $1$ without loss of generality, at the cost of changing $\alpha$ by one of the values in $\{\alpha,-1-\alpha,\tfrac{1}{\alpha},\tfrac{-1}{1+\alpha},
\tfrac{-\alpha}{1+\alpha},\tfrac{1+\alpha}{-\alpha}\}$.

Also, we may scale $h$ so that there is a symplectic basis $\{u_1,v_1\}$ of $V_1$ with $h(u_1)=u_1$ and $h(v_1)=-v_1$. In particular the eigenvalues of $h$ on $V_1$ are $1$ and $-1$, and hence the action of $h$ induces a $5$-grading on $\frg$: $\frg=\frg_{-2}\oplus\frg_{-1}\oplus\frg_0\oplus\frg_1\oplus\frg_2$, as in \textbf{\ref{ss:F(4)toralenA1}}. Since the degree of $h$ is $0$, each $\frg_m$, $-2\leq m\leq 2$, is a graded subspace in our grading $\Gamma$, so our fine grading $\Gamma$ is a refinement of this $5$-grading, and hence the subgroup
\[
\left\{\iota_{(d_\mu,I_{V_2},I_{V_3})}:\mu\in\bF^\times\right\}
\]
is contained in $D$. (Here $d_\mu$ denotes the element in $Sp(V_l)$ with coordinate matrix $\spmatrix{\mu}$ in our symplectic basis $\{u_l,v_l\}$, $l=1,2,3$.)

\smallskip

In case the restriction of $D$ to each $V_l$, $l=1,2,3$, does not induce a $\bZ_2^2$-grading, the argument above shows that the subgroup (for a suitable election of symplectic bases)
\[
\left\{\iota_{(d_{\mu_1},d_{\mu_2},d_{\mu_3})}: \mu_1,\mu_2,\mu_3\in\bF^\times\right\}
\]
is contained in $D$. But this subgroup is the maximal abelian diagonalizable subgroup of $\Aut\frg$ which induces the fine $\bZ^3$-grading given by the root space decomposition relative to a Cartan subalgebra. Hence $D$ equals this subgroup, and we get the fine $\bZ^3$-grading (Cartan grading) whose type is again $(14,0,1)$, the three-dimensional homogeneous space being the zero space, which is the Cartan subalgebra involved.

\medskip

\subsection{}\label{ss:Gamma12Z3Z22} It cannot be the case that $D$ induces a $\bZ_2^2$-grading only on one of the ideals $\frsp(V_l)$. Indeed, we may assume that $l=3$. The arguments above give that we may choose suitable bases such that the subgroup
\[
\left\{\iota_{(d_{\mu_1},d_{\mu_2},I_{V_3})}: \mu_1,\mu_2\in\bF^\times\right\}
\]
is contained in $D$, and hence $D$ is contained in the centralizer of this subgroup in $\Int\frg$, which is the subgroup
\[
\left\{\iota_{(d_{\mu_1},d_{\mu_2},f_3)}: \mu_1,\mu_2\in\bF^\times,\, f_3\in Sp(V_3)\right\}.
\]
The fact that the grading induced on $\frsp(V_3)$ is a $\bZ_2^2$-grading shows that we may choose a symplectic basis on $V_3$ such that $D$ contains elements
\[
\iota_{(d_{\mu_1},d_{\mu_2},a_3)},\quad
\iota_{(d_{\nu_1},d_{\nu_2},b_3)}
\]
for some nonzero scalars $\mu_1,\mu_2,\nu_1,\nu_2$. But these elements do not commute (they anticommute), a contradiction since $D$ is an abelian subgroup.

\medskip

\subsection{}\label{ss:Gamma1Z23Z22} We are then left with the situation in which $D$ induces a $\bZ_2^2$-grading on two of the simple ideals in $\frg\subo$, while the other simple ideal contains a nonzero diagonalizable element of degree $0$. At the cost of changing the value of $\alpha$, we may assume that $D$ induces a $\bZ_2^2$-grading on $\frsp(V_2)$ and $\frsp(V_3)$.

As before we have that the subgroup
\[
\left\{\iota_{(d_{\mu},I_{V_2},I_{V_3})}:\mu\in\bF^\times\right\}
\]
is contained in $D$, that $D$ is contained in its centralizer in $\Int \frg$, and that we may choose symplectic bases on $V_2$ and $V_3$ such that, denoting by $Q_8^l$ the subgroup of $Sp(V_l)$ generated by $a_l$ and $b_l$ in \eqref{eq:ajbj}, $D$ is contained in the subgroup
\[
\left\{\iota_{(d_{\mu},f_2,f_3)}:\mu\in\bF^\times,\,f_2\in Q_8^2,\, f_3\in Q_8^3\right\}.
\]
This last group is isomorphic to the group $(\bF^\times\times Q_8\times Q_8)/K$ considered in Proposition \ref{pr:FxQ82K}. This proposition shows that we may choose symplectic bases on $V_2$ and $V_3$ such that $D$ becomes the subgroup generated by the elements
\begin{equation}\label{eq:iotamuII}
\iota_{(d_{\mu},I_{V_2},I_{V_3})},\ \mu\in\bF^\times,\quad
\iota_{(I_{V_1},a_2,a_3)},\ \iota_{(I_{V_1},b_2,b_3)},
\end{equation}
which is isomorphic to $\bF^\times\times \bZ_2^2$. Therefore we get a $\bZ\times \bZ_2^2$-grading. Its type is easily checked to be $(11,3)$, where the three two-dimensional homogeneous spaces are contained in $\frsp(V_2)\oplus\frsp(V_3)$.

If $\alpha\ne -\tfrac{1}{2}$, then the centralizer of $D$ is contained in $\Int\frg$, so if $D$ were contained in a larger abelian diagonalizable group $\tilde D$, this would be contained in $\Int \frg$, so it would induce $\bZ_2^2$-gradings on $\frsp(V_2)$ and $\frsp(V_3)$. Since $\iota_{(d_{\mu},I_{V_2},I_{V_3})}$, $\mu\in\bF^\times$, belongs to $\tilde D$ too, it induces a $\bZ$-grading on $\frsp(V_1)$. Hence $\tilde D$ is contained in the subgroup isomorphic to $(\bF^\times\times Q_8\times Q_8)/K$ considered in Proposition \ref{pr:FxQ82K}, but $D$ is a maximal abelian subgroup here.

On the other hand, if $\alpha=-\tfrac{1}{2}$, then the automorphism $\pi_{2,3}$ which permutes $V_2$ and $V_3$ (once we fix the symplectic bases as before) centralizes the action of $D$, and hence $D$ is not a maximal abelian diagonalizable subgroup of $\Aut\frg$ in this case.

\begin{remark}\label{re:alpha-12}
Even for $\alpha=-\frac12$, the group generated by the elements $\iota_{(I_{V_1},I_{V_2},d_{\mu})}$ ($\mu\in\bF^\times)$,
$\iota_{(a_1,a_2,I_{V_3})}$ and $\iota_{(b_1,b_2,I_{V_3})}$ is a maximal abelian diagonalizable subgroup of $\Aut\frg$, and hence it induces a fine grading over $\bZ\times\bZ_2^2$. This corresponds to the fine $\bZ\times \bZ_2^2$-grading on $D(2,1;1)$ (or $D(2,1;-2)$) given by the subgroup generated by the elements in \eqref{eq:iotamuII}.

Also, for $\alpha\not\in\{1,-2,-\frac12,\omega,\omega^2\}$ there appear three non-equivalent fine $\bZ\times\bZ_2^2$-gradings, as the simple ideal of $\frg\subo$ on which the induced grading is a $\bZ$-grading may be any of the three simple ideals, and these ideals are not permuted by the automorphism group $\Aut\frg$.
\end{remark}

\bigskip


\section{Outer fine gradings on $D(2,1;\alpha)$}\label{se:outerD21a}

This section is devoted to the classification of the fine gradings on the simple Lie superalgebra $\frg=D(2,1;\alpha)$ whose diagonal group is not contained in $\Int\frg$. Therefore we have to deal with two cases, either $\alpha$ is a primitive cubic root $\omega$ of $1$, or $\frg$ is the orthosymplectic Lie superalgebra $\frosp(4,2)$.

\subsection{Case $D(2,1;\omega)$}\label{ss:D21omega}\quad Let $\frg$ be the Lie superalgebra $D(2,1;\omega)$, where $\omega$ is a fixed primitive cubic root of $1$. Given symplectic isomorphisms $f_l:V_l\rightarrow V_{l+1}$ (that is, linear isomorphisms such that $b_{l+1}\bigl(f_l(u_l),f_l(v_l)\bigr)=b_l(u_l,v_l)$ for any $u_l,v_l\in V_l$), for $l=1,2,3$, where we are taking the indices modulo $3$, consider the linear map $\Phi_{(f_1,f_2,f_3)}:\frg\rightarrow\frg$ given by:
\begin{equation}\label{eq:Phif123}
\begin{split}
\Phi_{(f_1,f_2,f_3)}(x_1,x_2,x_3)&
   =\bigl(f_3x_3f_3^{-1},f_1x_1f_1^{-1},f_2x_2f_2^{-1}\bigr),\\
\Phi_{(f_1,f_2,f_3)}(u_1\otimes u_2\otimes u_3)&=\omega f_3(u_3)\otimes f_1(u_1)\otimes f_2(u_2),
\end{split}
\end{equation}
for $x_l\in \frsp(V_l)$, $u_l\in V_l$, $l=1,2,3$.

This map $\Phi_{(f_1,f_2,f_3)}$ is an automorphism of $\frg$ which permutes cyclically the three simple ideals $\frsp(V_l)$ of $\frg\subo$.

\begin{lemma}\label{le:cyclicauto}
The automorphisms of $\frg$ permuting cyclically the three simple ideals of $\frg\subo$: $\frsp(V_1)\mapsto\frsp(V_2)\mapsto\frsp(V_3)\mapsto\frsp(V_1)$, are precisely the automorphisms $\Phi_{(f_1,f_2,f_3)}$ defined in \eqref{eq:Phif123}.
\end{lemma}
\begin{proof}
If $\varphi$ is such an automorphism, then for fixed $f_1,f_2,f_3$ as above, $\varphi\Phi_{(f_1,f_2,f_3)}^{-1}$ is an automorphism which fixes the three simple ideals of $\frg\subo$, and hence belongs to $\Int\frg$. Hence there are elements $g_l\in Sp(V_l)$, $l=1,2,3$, such that $\varphi\Phi_{(f_1,f_2,f_3)}^{-1}=\iota_{(g_1,g_2,g_3)}$. But then,
\[
\varphi=\iota_{(g_1,g_2,g_3)}\Phi_{(f_1,f_2,f_3)}=\Phi_{(g_2f_1,g_3f_2,g_1f_3)}
\]
as required. (Note that it is enough to check the above equality in elements of $\frg\subuno$, where it is obvious, as $\frg\subuno$ generates $\frg$.)
\end{proof}

Let $\varphi=\Phi_{(f_1,f_2,f_3)}$ be an automorphism as before which permutes cyclically the simple ideals of $\frg\subo$. Use $f_1$ and $f_2$ to identify $V_2$ and $V_3$ with $V=V_1$. Then with $f=f_3f_2f_1\in Sp(V)$, we have $\varphi=\Phi_{(I_V,I_V,f)}$.

\begin{lemma}\label{le:centralizervarphi}
Under the above conditions, the centralizer of $\varphi$ in $\Int\frg$ is the subgroup
\[
\{\iota_{(g,g,g)}: g\in Sp(V),\ fg=gf\},
\]
which is naturally isomorphic to the centralizer of $f$ in $Sp(V)$.
\end{lemma}
\begin{proof}
For $g_l\in Sp(V)$, $l=1,2,3$,
\[
\begin{split}
\iota_{(g_1,g_2,g_3)}\varphi&=\Phi_{(g_2,g_3,g_1f)},\\
\varphi\iota_{(g_1,g_2,g_3)}&=\Phi_{(g_1,g_2,fg_3)},
\end{split}
\]
and hence, the automorphism $\iota_{(g_1,g_2,g_3)}$ centralizes $\varphi$ if and only if there are $\epsilon_l=\pm 1$, $l=1,2,3$, with $\epsilon_1\epsilon_2\epsilon_3=1$ such that $g_1=\epsilon_1g_2$, $g_2=\epsilon_2g_3$ and $fg_3=\epsilon_3g_1f$. It follows that $g_3f=fg_3$, and $(g_1,g_2,g_3)=(\epsilon_3g_3,\epsilon_2g_3,g_3)=(\epsilon_2 g,\epsilon_3 g,\epsilon_1 g)$, with $g=\epsilon_1g_3\in Sp(V)$. The result follows since $\iota_{(\epsilon_2 g,\epsilon_3 g,\epsilon_1 g)}=\iota_{(g,g,g)}$.
\end{proof}

\smallskip

Let $\Gamma:\frg=\oplus_{g\in G}\frg_g$ be a fine grading and let $D=\Diag_\Gamma(\frg)$ be the diagonal group of the grading. Assume that $D$ contains an automorphism $\varphi$ as above which permutes cyclically the three simple ideals of $\frg\subo$, so that with the identifications above we may assume that $\varphi=\Phi_{(I_V,I_V,f)}$ for some $f\in Sp(V)$. Note that $\varphi^3=\iota_{(f,f,f)}$, so that, since $\varphi$ is diagonalizable, so is $f$. Hence there is a symplectic basis $\{u,v\}$ of $V$ such that the coordinate matrix of $f$ is $d_\mu=\spmatrix{\mu}$ for some $\mu\in\bF^\times$.

The centralizer of $f$ in $Sp(V)$ is then easily checked to be either the whole $Sp(V)$ if $\mu=\pm 1$, or the subgroup $\{d_\nu:\nu\in\bF^\times\}$ otherwise (where we identify any element in $Sp(V)$ with its coordinate matrix in the basis above). In the second case, our maximal abelian diagonalizable subgroup $D$ is contained in the centralizer of $\varphi$, which is abelian and diagonalizable.
While in the first case, the centralizer of $f$ in $Sp(V)$ is the whole $Sp(V)$. The maximality of $D$ shows that there is some diagonalizable $g\in Sp(V)$, $g\ne\pm I_V$ such that $\iota_{(g,g,g)}\in D$. Now we take a symplectic basis $\{u,v\}$ of $V$ in which the coordinate matrix of $g$ is $d_\mu$ for some $\mu\ne \pm 1$, and we conclude (since $D$ is contained in the centralizer of $\iota_{(g,g,g)}$) that $D$ is
generated by $\varphi$ and the subgroup $\{\iota_{(d_\nu,d_\nu,d_\nu)}:\nu\in \bF^\times\}$.

In any case, there is an element $g\in Sp(V)$ such that $\iota_{(g,g,g)}\in D$ and $g^3=f^{-1}$. Then $\iota_{(g,g,g)}\varphi=\Phi_{(g,g,gf)}$, which is an automorphism in $D$ whose cube is $\iota_{(I_V,I_V,I_V)}=1$, as $g^3f=I_V$. Therefore, we may assume from the beginning that our element $\varphi$ in $D$ which permutes cyclically the three simple ideals of $\frg\subo$ satisfies $\varphi^3=1$, and hence we may choose the identifications of $V_2$ and $V_3$ with $V=V_1$ so that $\varphi=\Phi_{(I_V,I_V,I_V)}$, and
\[
D=\{\iota_{(d_\mu,d_\mu,d_\mu)}:\mu\in\bF^\times\}\times\langle\varphi\rangle\cong \bF^\times \times \bZ_3.
\]
We thus obtain a $\bZ\times\bZ_3$-grading.

It is straightforward to check that all the homogeneous spaces are one-dimensional. That is, its type is $(17)$.

\bigskip

\subsection{Case $D(2,1;-\tfrac{1}{2})$}\label{ss:D21osp42}\quad In this subsection we deal with the case in which the diagonal subgroup of our fine grading permutes two of the three simple ideals in $\frg\subo$. Then $\alpha\in\{1,-\tfrac{1}{2},-2\}$. Without loss of generality we may assume that $\alpha=-\tfrac{1}{2}$ and hence that the permuted ideals are $\frsp(V_2)$ and $\frsp(V_3)$. In this case $\frg=D(2,1;-\tfrac{1}{2})$ is isomorphic to the orthosymplectic Lie superalgebra $\frosp(4,2)$.

\smallskip

Working as in \textbf{\ref{ss:D21omega}}, given symplectic isomorphisms $f:V_1\rightarrow V_1$, $g:V_2\rightarrow V_3$ and $h:V_3\rightarrow V_2$, the linear map $\hat\Phi_{(f,g,h)}:\frg\rightarrow\frg$ given by:
\begin{equation}\label{eq:hatPhifgh}
\begin{split}
\hat\Phi_{(f,g,h)}(x_1,x_2,x_3)&
   =\bigl(fx_1f^{-1},hx_3h^{-1},gx_2g^{-1}\bigr),\\
\hat\Phi_{(f,g,h)}(u_1\otimes u_2\otimes u_3)&= f(u_1)\otimes h(u_3)\otimes g(u_2),
\end{split}
\end{equation}
for $x_l\in \frsp(V_l)$, $u_l\in V_l$, $l=1,2,3$,
is an automorphism of $\frg$ which permutes $\frsp(V_2)$ and $\frsp(V_3)$.

We easily prove the following result, similar to Lemma \ref{le:cyclicauto}.

\begin{lemma}\label{le:permut23auto}
The automorphisms of $\frg$ permuting the simple ideals $\frsp(V_2)$ and $\frsp(V_3)$ of $\frg\subo$ are precisely the automorphisms $\hat\Phi_{(f,g,h)}$ defined in \eqref{eq:hatPhifgh}.
\end{lemma}

Let $\varphi=\hat\Phi_{(f,g,h)}$ be an automorphism as before which permutes the two last simple ideals of $\frg\subo$.

\begin{lemma}\label{le:centralizervarphibis}
Under the above conditions, the centralizer of $\varphi$ in $\Int\frg$ is the subgroup
\[
\begin{split}
\{\iota_{(f_1,f_2,h^{-1}f_2h)}&: f_1\in Sp(V_1),\ f_2\in Sp(V_2),\\
    &\quad f_1f=\epsilon ff_1,\ f_2hg=\epsilon hgf_2,\ \text{for some $\epsilon=\pm 1$}\}.
\end{split}
\]
\end{lemma}
\begin{proof}
For $f_l\in Sp(V_l)$, $l=1,2,3$,
\[
\begin{split}
\iota_{(f_1,f_2,f_3)}\varphi&=\hat\Phi_{(f_1f,f_3g,f_2h)},\\
\varphi\iota_{(f_1,f_2,f_3)}&=\hat\Phi_{(ff_1,gf_2,hf_3)},
\end{split}
\]
and hence, the automorphism $\iota_{(f_1,f_2,f_3)}$ centralizes $\varphi$ if and only if there are $\epsilon_l=\pm 1$, $l=1,2,3$, with $\epsilon_1\epsilon_2\epsilon_3=1$ such that $f_1f=\epsilon_1ff_1$, $f_3g=\epsilon_2gf_2$ and $f_2h=\epsilon_3hf_3$. But $\iota_{(f_1,f_2,f_3)}=\iota_{(-f_1,f_2,-f_3)}$, so we may assume $\epsilon_3=1$ and $\epsilon_1=\epsilon_2=\epsilon$. The result follows.
\end{proof}

\smallskip

Let $\Gamma:\frg=\oplus_{g\in G}\frg_g$ be a fine grading and let $D=\Diag_\Gamma(\frg)$ be the diagonal group of the grading. Assume that $D$ contains an automorphism $\varphi=\hat\Phi_{(f,g,h)}$ as above which permutes $\frsp(V_2)$ and $\frsp(V_3)$, and write $\tilde D=D\cap \Int\frg$.  Note that $\varphi^2=\iota_{(f^2,hg,gh)}$, so that, since $\varphi$ is diagonalizable, so are $f^2$ and $hg$, and hence so is $f$ too.

\smallskip

For $l=1,2,3$ we have the projection map:
\[
\begin{split}
\pi_l:\Int\frg&\rightarrow PSL(V_l)\cong\Aut\frsp(V_l)\\
\iota_{(f_1,f_2,f_3)}&\mapsto\ \bar f_l,
\end{split}
\]
where $\bar f_l$ denotes the class of $f_l\in Sp(V_l)=SL(V_l)$ in $PSL(V_l)$.

For $l=1$, we may even consider the projection map $\pi_1:\Aut\frg\rightarrow PSL(V_1)$, since any automorphism of $\frg$ preserves $\frsp(V_1)$.

The abelian diagonalizable subgroup $\pi_1(D)$ induces a grading on $\frsp(V_1)$, while $\pi_2(\tilde D)$ and $\pi_3(\tilde D)$ induce gradings on $\frsp(V_2)$ and $\frsp(V_3)$. Moreover, Lemma \ref{le:centralizervarphibis} shows that $\pi_3(\tilde D)$ is isomorphic to $\pi_2(\tilde D)$. From the known facts about gradings on $\frsl_2(\bF)$, either there exists a diagonalizable element  in the zero component of the grading induced by $\pi_2(\tilde D)$ on $\frsp(V_2)$, or $\pi_2(\tilde D)$ (and hence $\pi_3(\tilde D)$ too) is isomorphic to $\bZ_2^2$. Actually, in this case a symplectic basis of $V_2$ can be chosen such that $\pi_2(\tilde D)=\langle \bar a_2,\bar b_2\rangle$, with $a_2$ and $b_2$ as in \eqref{eq:ajbj}.

We are left with several possibilities:

\begin{enumerate}
\item Assume first that for the grading induced by $\pi_1(D)$ on $\frsp(V_1)$ there is a diagonalizable element in the zero component. Then, as in \textbf{\ref{ss:Gamma123Z}}, there is a symplectic basis $\{u_1,v_1\}$ of $V_1$ such that $\{\iota_{(d_\mu,I_{V_2},I_{V_3})}:\mu\in\bF^\times\}$ is contained in $\tilde D$ and $\pi_1(D)=\{\bar d_\mu:\mu\in\bF^\times\}$. In particular, with $\varphi=\hat\Phi_{(f,g,h)}$, it follows that $f=d_\mu$ for some $\mu\in\bF^\times$. We may change $\varphi$ to $\iota_{(d_{\mu^{-1}},I_{V_2},I_{V_3})}\varphi$ and hence assume $f=I_{V_1}$. Then the centralizer of $\varphi$ in $\Int\frg$ is the subgroup
    \[
    \{\iota_{(f_1,f_2,h^{-1}f_2h)}: f_1\in Sp(V_1),\ f_2\in Sp(V_2),\ f_2hg=hgf_2\}.
    \]
    Since $hg$ is diagonalizable, we may take a symplectic basis $\{u_2,v_2\}$ of $V_2$ such that the coordinate matrix of $hg$ is $d_\xi$ for some $\xi\in \bF^\times$.

    In case $hg\ne \pm I_{V_2}$, then $\xi\ne \pm 1$, and the centralizer of $d_\xi$ in $Sp(V_2)$ is the subgroup $\{d_\mu:\mu\in\bF^\times\}$.
    It follows that $\tilde D$ is contained (since $\pi_1(D)=\{\bar d_\mu:\mu\in\bF^\times\}$) in the abelian subgroup
    \[
    \{\iota_{(d_\mu,d_\nu,h^{-1}d_\nu h)}:\mu,\nu\in\bF^\times\}
    \]
    of the centralizer of $\varphi$, and hence, by maximality of $D$, this subgroup is the whole $\tilde D$. Consider now the symplectic basis already taken in $V_1$ and $V_2$ and the symplectic basis $\{u_3=g(u_2),v_3=g(v_2)\}$ of $V_3$. In these bases the coordinate matrix of $g$ is the identity, while the coordinate matrix of $h$ is $d_\xi=\spmatrix{\xi}$. We may change $\varphi$ to
    \[
    \iota_{(I_{V_1},d_{\xi^{-1/2}},d_{\xi^{-1/2}})}\varphi=
    \hat\Phi_{(I_{V_1},d_{\xi^{-1/2}},d_{\xi^{1/2}})},
    \]
    whose square is the identity. That is, we may assume that $\varphi$ equals the map $\hat\Phi_{(I_{V_1},\tilde g,\tilde h)}$, with $\tilde g:V_2\rightarrow V_3$ such that $\tilde g(u_2)=\xi^{-1/2}u_3$, $\tilde g(v_2)=\xi^{1/2}v_3$, and $\tilde h:V_3\rightarrow V_2$ such that $\tilde h(u_3)=\xi^{1/2}u_2$, $\tilde h(v_3)=\xi^{-1/2}v_2$. If we take the new symplectic basis $\{\xi^{-1/2}u_3,\xi^{1/2}v_3\}$ of $V_3$,  the coordinate matrix of both $\tilde g$ and $\tilde h$ in these bases is the identity matrix, and
    \[
    D=\tilde D\times \langle\varphi\rangle=\{\iota_{(d_\mu,d_\nu,d_\nu)}: \mu,\nu\in\bF^\times\}\times \langle\varphi\rangle\cong(\bF^\times)^2\times \bZ_2.
    \]
    We get then a $\bZ^2\times\bZ_2$-grading on $\frg$.

    The first copy of $\bZ$ gives the $\bZ$-grading whose homogeneous spaces are the eigenspaces of the adjoint action of $z=\left(\begin{smallmatrix}1&0\\ 0&-1\end{smallmatrix}\right)$ in $\frsp(V_1)$ (a $5$-grading). The second copy of $\bZ$ gives the $\bZ$-grading whose homogeneous spaces are the eigenspaces of the adjoint action of $(z,z)\in\frsp(V_2)\oplus\frsp(V_3)$. The copy of $\bZ_2$ corresponds to the $\bZ_2$-grading on $\frg$ given by the automorphism $\varphi$ which fixes elementwise $V_1$ and interchanges $V_2$ and $V_3$. The type of the $\bZ^2\times\bZ_2$-grading is $(15,1)$, the two-dimensional homogeneous space being the subspace generated by $z\in\frsp(V_1)$ and $(z,z)\in\frsp(V_2)\oplus\frsp(V_3)$. This grading is certainly fine.

    \smallskip

    In case $hg=-I_{V_2}$, then we may compose $\varphi$ with the grading automorphism $\iota_{(-I_{V_1},I_{V_2},I_{V_3})}=\iota_{(I_{V_1},-I_{V_2},I_{V_3})}$ (which belongs to $D$), and hence assume that $hg=I_{V_2}$. In this latter case, $\varphi^2=1$, and the centralizer of $\varphi$ in $\Int\frg$ is the subgroup
    \[
    \{\iota_{(f_1,f_2,h^{-1}f_2h)}: f_1\in Sp(V_1),f_2\in Sp(V_2)\}.
    \]
    This group is isomorphic to the group $Sp(V_1)\times PSL(V_2)$ by means of the map $\iota_{(f_1,f_2,h^{-1}f_2h)}\mapsto (f_1,\bar f_2)$.

    Our assumptions on $\pi_1(D)$ show that in this case $\tilde D$ is contained in the group
    \[
    \{\iota_{(d_\mu,f,h^{-1}fh)}: \mu\in\bF^\times,\, f\in Sp(V_2)\}\cong\bF^\times\times PSL(V_2)\cong\bF^\times\times \Aut\frsp(V_2).
    \]
    The maximality of $D$ shows that $\tilde D$ is a maximal abelian diagonalizable subgroup of this group, and hence either there is a symplectic basis $\{u_2,v_2\}$ such that
    \[
    \tilde D=\{\iota_{(d_\mu,d_\nu,h^{-1}d_\nu h)}: \mu,\nu\in\bF^\times\},
    \]
    and we are in the previous situation giving the $\bZ^2\times\bZ_2$-grading, or $\pi_2(\tilde D)$ induces a $\bZ_2^2$-grading on $\frsp(V_2)$. In this latter case, there is a symplectic basis $\{u_2,v_2\}$ of $V_2$ such that
    \[
    \tilde D=\left\{\iota_{(d_\mu,f,h^{-1}fh)}: \mu\in\bF^\times,\, f\in\{I_{V_2},a_2,b_2,a_2b_2\}\right\}\cong\bF^\times\times\bZ_2^2.
    \]
    Therefore $D=\tilde D\times \langle\varphi\rangle\cong\bF^\times\times \bZ_2^3$ and we get a $\bZ\times\bZ_2^3$-grading on $\frg$. It is easy to check that its type is $(17)$.

    At this point we should look back to Subsection~\textbf{\ref{ss:Gamma1Z23Z22}}. The grading above refines the grading over $\bZ\times\bZ_2^2$ considered there ($\alpha=-\frac12$).

\smallskip

\item Otherwise $\Gamma$ induces a $\bZ_2^2$-grading on $\frsp(V_1)$. Let us check that in this case we have the following conditions:
    \begin{equation}\label{eq:osp42pis}
    \pi_1(D)\cong\bZ_2^2,\ \pi_1(\tilde D)\cong\bZ_2,\ \pi_2(\tilde D)\cong\bZ_2^2\cong\pi_3(\tilde D).
    \end{equation}
    Actually, the fact that $\pi_1(D)$ be isomorphic to $\bZ_2^2$ is equivalent to $\Gamma$ inducing a $\bZ_2^2$-grading on $\frsp(V_1)$. Then there is a symplectic basis $\{u_1,v_1\}$ of $V_1$ such that $\pi_1(D)=\langle \bar a_1,\bar b_1\rangle$ ($a_1$ and $b_1$ as in \eqref{eq:ajbj}). If $\pi_1(\tilde D)$ were isomorphic to $\bZ_2^2$, then, as $\tilde D$  induces a grading on $\frg$ preserving the three simple ideals of $\frg\subo$, the arguments in \textbf{\ref{ss:Gamma12Z3Z22}} would show that $\pi_2(\tilde D)\cong\bZ_2^2\cong\pi_3(\tilde D)$. But for any $\iota_{(f_1,f_2,h^{-1}f_2h)}$ and $\iota_{(\tilde f_1,\tilde f_2,h^{-1}\tilde f_2h)}$ in $\tilde D$, $f_1,\tilde f_1\in \{\pm I_{V_1},\pm a_1,\pm b_1,\pm a_1b_1\}$. Then, as $\pi_2(\tilde D)\cong\bZ_2^2$, $f_2\tilde f_2=\epsilon \tilde f_2f_2$ for $\epsilon=\pm 1$, and hence $(h^{-1}f_2h)(h^{-1}\tilde f_2h)=\epsilon(h^{-1}\tilde f_2h)(h^{-1}f_2h)$. Thus, since $\iota_{(f_1,f_2,h^{-1}f_2h)}$ and $\iota_{(\tilde f_1,\tilde f_2,h^{-1}\tilde f_2h)}$ commute, we conclude $f_1\tilde f_1=\tilde f_1f_1$, and hence $\pi_1(\tilde D)$ is a cyclic group of order $2$ ($f_1$ and $\tilde f_1$ are elements of a quaternion group which commute, and hence $\pi_1(\tilde D)$ is the image  of a commutative subgroup of the quaternion group modulo the center).

    Note that the element $\bar f$ is in $\pi_1(D)$ (recall $\varphi=\hat\Phi_{(f,g,h)}$), and $\pi_1(D)=\langle \bar f,\pi_1(\tilde D)\rangle$, so $f\ne\pm 1$ and we may take a symplectic basis $\{u_1,v_1\}$ of $V_1$ such that $f=b_1$ and $\pi_1(\tilde D)=\langle \bar a_1\rangle$. Thus there is an element of the form $\iota_{(\pm a_1,f_2,h^{-1}f_2h)}$ in $\tilde D$. Since $\pm a_1$ anticommutes with $b_1$ and $\iota_{(\pm a_1,f_2,h^{-1}f_2h)}$ is in the centralizer of $\varphi$, it follows that $f_2$ anticommutes with $hg$. Also $\varphi^2=\hat\Phi_{(f,g,h)}^2=\iota_{(f^2,hg,gh)}$ so $\overline{hg}\in\pi_2(\tilde D)$. Thus $\pi_2(\tilde D)=\langle \bar f_2,\overline{hg}\rangle\cong\bZ_2^2\cong\pi_3(\tilde D)$. This finishes the proof of the conditions in \eqref{eq:osp42pis}.

    \smallskip

    Therefore, we may choose a symplectic basis $\{u_2,v_2\}$ of $V_2$ such that we have
    \begin{equation}\label{eq:Dtildeotro}
    \tilde D\subseteq \left\{\iota_{(f_1,f_2,h^{-1}f_2h)}: f_1\in\{\pm I_{V_1},\pm a_1\},\ f_2\in\{\pm I_{V_2},\pm a_2,\pm b_2,\pm a_2b_2\}\right\},
    \end{equation}
    and $\varphi=\hat\Phi_{(b_1,g,h)}$, $hg=b_2$.

    But the intersection of the subgroup in the right hand side of Equation~\eqref{eq:Dtildeotro} with the centralizer of $\varphi$ is
    \[
    \begin{split}
    &\left\{\iota_{(\pm I_{V_1},f_2,h^{-1}f_2h)}: f_2\in\{\pm I_{V_1},\pm b_2\}\right\}\bigcup\left\{\iota_{\pm a_1,f_2,h^{-1}f_2h)}: f_2\in\{\pm a_2,\pm a_2b_2\}\right\}\\
    &\quad =
    \left\{\iota_{(\pm I_{V_1},f_2,h^{-1}f_2h)}: f_2\in\{I_{V_1},b_2\}\right\}\bigcup\left\{\iota_{\pm a_1,f_2,h^{-1}f_2h)}: f_2\in\{a_2,a_2b_2\}\right\},
    \end{split}
    \]
    which is an abelian subgroup. By maximality, $\tilde D$ coincides with this subgroup, and $D$ is generated by $\tilde D$ and $\varphi$.

    Note that $\varphi^2=\hat\Phi_{(b_1,g,h)}^2=\iota_{(b_1^2,hg,gh)}
    =\iota_{(-I_{V_1},b_2,h^{-1}b_2h)}$ and $\varphi^4=1$. Also, consider the element $\psi=\iota_{(a_1,a_2,h^{-1}a_2h)}$ in $\tilde D$. Then $\psi^2=\iota_{(-I_{V_1},I_{V_2},I_{V_3})}$ is the grading automorphism, so $\psi^4=1$ too. Besides $\langle \varphi\rangle\cap\langle\psi\rangle=1$. Hence, since $\tilde D$ has eight elements, and $[D:\tilde D]=[D:D\cap\Int \frg]=2$, it follows that $\lvert D\rvert=16$, and thus $D=\langle \varphi\rangle\times\langle\psi\rangle\cong\bZ_4^2$.

    Then we get a fine $\bZ_4^2$-grading whose homogeneous spaces are the common eigenspaces of $\varphi$ and $\psi$. Its type is easily computed to be $(13,2)$, the two homogeneous spaces of dimension two being contained in $\frg\subo$.
\end{enumerate}

Let us put together the fine gradings on $D(2,1;\alpha)$ in the final result of the paper:

\begin{theorem}\label{th:MainThD21a}
The fine gradings on the simple Lie superalgebra $D(2,1;\alpha)$ ($\alpha\ne 0,-1$) are, up to equivalence, the following:
\begin{romanenumerate}
\item The Cartan grading over $\bZ^3$ of type $(14,0,1)$.
\item A $\bZ_4\times\bZ_2^2$-grading of type $(14,0,1)$.
\item For $\alpha\not\in\{1,-2,-\frac12,\omega,\omega^2\}$, three non-equivalent $\bZ\times\bZ_2^2$-gradings of type $(11,3)$, and for $\alpha\in\{1,-2,-\frac12,\omega,\omega^2\}$ a unique $\bZ\times\bZ_2^2$-grading of type $(11,3)$.
\item For $\alpha\in\{\omega,\omega^2\}$, a $\bZ\times\bZ_3$-grading of type $(17)$.
\item For $\alpha\in\{1,-2,-\frac12\}$, a $\bZ\times\bZ_2^3$-grading of type $(17)$.
\item For $\alpha\in\{1,-2,-\frac12\}$, a $\bZ^2\times \bZ_2$-grading of type $(15,1)$.
\item For $\alpha\in\{1,-2,-\frac12\}$, a $\bZ_4\times\bZ_4$-grading of type $(13,2)$.
\end{romanenumerate}
\end{theorem}

\end{document}